\documentclass{article}

\usepackage[utf8]{inputenc}
\usepackage{amsmath,amsfonts,amssymb}
\usepackage{amsthm,upref}
\usepackage[a4paper]{geometry}
\usepackage{color} 
\usepackage[dvipsnames]{xcolor}
\usepackage{graphicx}
\usepackage{subfigure}
\usepackage{tabularx,float}
\usepackage[hidelinks]{hyperref}
\usepackage{comment} 
\usepackage{tabularx} 
\usepackage{titlefoot} 
\usepackage{fancyhdr}

\usepackage{fullpage}

\usepackage{todonotes}
\usepackage{color} 
\newcommand{\note}[1]{ \textcolor{red}{ \ [ \ #1 \  ] \ }} 
\newcommand\SJ[1]{{\color{black}{#1}}} 
\newcommand\KUK[1]{{\color{black}{#1}}} 
\newcommand\PS[1]{{\color{black}{#1}}}

\newtheorem{thm}{Theorem}[section]

\newtheorem{proposition}[thm]{Proposition}

\newtheorem{lemma}[thm]{Lemma}
\newtheorem{remark}[thm]{Remark}

\newcommand{\eps}{\epsilon}

\newcommand{\R}{\mathbb{R}}

\newcommand{\dd}{{\mathrm d}}
\newcommand{\transpose}{\textnormal{T}}
\newcommand{\e}{{\mathrm e}}

\title{Travelling Waves and Exponential Nonlinearities in the Zeldovich-Frank-Kamenetskii Equation}
\author{S.~Jelbart, K.~U.~Kristiansen \& P.~Szmolyan}
\date{\today}

\begin{document}
	\maketitle
	
	\begin{abstract}
		We prove the existence of a family of travelling wave solutions in a variant of the \textit{Zeldovich-Frank-Kamenetskii (ZFK) equation}, a reaction-diffusion equation which models the propagation of planar {laminar premixed} flames in combustion theory. Our results are valid in an asymptotic regime which corresponds to a reaction with high activation energy, and provide a rigorous and geometrically informative counterpart to formal asymptotic results that have been obtained {for similar problems} using \textit{high activation energy asymptotics}. We also go beyond the existing results by (i) proving smoothness of the minimum wave speed function $\overline c(\eps)$, where $0<\eps \ll 1$ is the small parameter, and (ii) providing an asymptotic series for a flat slow manifold which plays a role in the construction of travelling wave solutions for non-minimal wave speeds $c > \overline c(\eps)$. The analysis is complicated by the presence of an exponential nonlinearity which leads to two different scaling regimes as $\eps \to 0$, which we refer to herein as the \textit{{convective}-diffusive} and \textit{diffusive-reactive} zones. The main idea of the proof is to use the geometric blow-up method to identify and characterise a $(c,\eps)$-family of heteroclinic orbits which traverse both of these regimes, and correspond to travelling waves in the original ZFK equation. More generally, our analysis contributes to a growing number of studies which demonstrate the utility of geometric blow-up approaches to the study dynamical systems with singular exponential nonlinearities.
	\end{abstract}
	
	\medskip
	
	

	\section{Introduction}
\label{sec:introduction}

Exponential nonlinearities appear \SJ{in} differential equation models across the natural sciences and beyond, ranging from models of earthquake faulting \cite{Bossolini2017b,Erickson2008,Kristiansen2019b} to electrical oscillators \cite{Hester1968,Jelbart2019c,LeCorbeiller1960} and plastic deformation in metals \cite{Brons2005,Estrin1980}. 
It is also common for modelers to introduce exponential nonlinearities for mathematical reasons, for example in the regularisation of non-smooth systems using a $\tanh$ function \cite{Kristiansen2017}, the study of gene regulatory networks which converge to a non-smooth approximating system at an exponential rate \cite{Glass2018,Ironi2011,Machina2013,Plahte2005}, or more recently, via the introduction of singular exponential coordinates in order to study {polynomial} systems using techniques based on tropical geometry \cite{Kristiansen2023,Portisch2017}.

All of the problems cited above share a common mathematical feature, namely, that the exponential nonlinearities lead to singular perturbation problems which have \PS{very different} 
dynamics \PS{and asymptotics} in different regions of phase space. 
The basic reason for this can be seen by considering the nonlinearity $\e^{x / \eps}$, which is exponentially small, $\mathcal O(1)$, or exponentially large as $\eps \to 0$ when $x < 0$, $x = \mathcal O(\eps)$, or $x > 0$ respectively. Thus, if $x$ is a {state variable} in a system {of} differential equations, then one expects to obtain {distinct dynamics for $x<0$, $x = 0$ and $x > 0$ as $\eps \to 0$}. 
Regions of phase space where the limit is not defined due to unbounded growth, in this case when $x > 0$, can often be controlled after a suitable `normalisation' (positive transformation of time). For example, the authors in \cite{Jelbart2019c} applied a normalisation which amounted to a division of the governing equations by $1 + \e^{x / \eps}$ in order to control the (otherwise unbounded) growth of the nonlinearity $\e^{x / \eps}$. 
Since
\[
\SJ{\frac{\e^{x/\eps}}{1+\e^{x/\eps}}\rightarrow
	\begin{cases}
		0 &x<0\\
		1 &x>0
	\end{cases}\quad \text{for}\quad \eps\to 0,}
\]
the normalised system is \textit{piecewise-smooth} in the limit $\eps \to 0$. 
From here, the loss of smoothness can be `resolved' via a geometric blow-up procedure which involves the insertion of an inner/switching layer which describes the dynamics for $x = \mathcal O(\eps)$, i.e.~`in between' the distinct regimes $x< 0$ and $x > 0$, using techniques that have in recent years been developed and applied in a large number of works, e.g.~in  \cite{bonet-rev2016a,Carvalho2011,Jelbart2021,Jelbart2020d,Kristiansen2019c,Kristiansen2015b,Kristiansen2015a,Kristiansen2019,Kristiansen2019d,Llibre2007}. 

In this work we show that similar techniques can also be used to describe travelling fronts in a reaction-diffusion equation known as the \textit{Zeldovich-Frank-Kamenetskii (ZFK)} equation, which arises in combustion theory as a model for the propagation of planar laminar premixed flames \cite{Buckmaster1982,Bush1979,Clavin2016,Williams1985,Zeldovich1938}. We focus on one of the simplest variants of the ZFK equation\SJ{:}
\begin{align}\label{eq:hallo}
\KUK{\theta_t = 
	\theta_{xx} +  \frac{\beta^2}{2}\theta (1 - \theta) \e^{- \beta (1- \theta)},}
	\end{align}
	\SJ{where $t \geq 0$, $x \in \R$ and $\theta = \theta(x,t) \in [0,1]$ is a dimensionless variable which is related to the temperature of the reaction. Equation \eqref{eq:hallo} appears in \cite{Clavin2016}; we refer to 
	Section \ref{sec:the_ZFK_equation} for details}.  The small parameter is $\eps := \beta^{-1}$, where $\beta$ is the so-called \textit{Zeldovich number}, a dimensionless quantity related to the activation energy of the Arrhenius reaction, and $\beta \to \infty$ ($\eps \to 0$) is the high activation energy limit. 
\KUK{In the pioneering work of Zeldovich and Frank-Kamenetskii \cite{Zeldovich1938}, the authors introduced the scalings that give rise to the specific (exponential) form  of the reaction term in \eqref{eq:hallo} and led to the development of \textit{high activation energy asymptotics (HAEA)}. Notice that in the limit of large activation energy ($\beta\rightarrow \infty$), the reaction term is exponentially small for $\theta<1$. 
A key observation of \cite{Zeldovich1938}, which forms the basis of HAEA, is that this leads to two different scaling regimes: a \textit{{convective}-diffusive zone} where the reaction term is exponentially small, and a \textit{reactive-diffusive zone} where a {convective} term which arises in the travelling wave problem can be neglected; we explain this in more details in the context of \eqref{eq:hallo} in Section \ref{sec:HAEA}. 
%
Their work was fundamental for many important works on pre-mixed flames which were based on HAEA, we refer to e.g.~\cite{Buckmaster1982,Bush1979,Clavin2016,Williams1985} and the many references therein.}

\KUK{Despite the success of HAEA in explaining numerous fundamental mechanisms in combustion processes, rigorous mathematical analysis seems to be lacking in this area. 
In this paper, we study travelling waves of \eqref{eq:hallo} using dynamical systems methods, in particular \textit{geometric singular perturbation theory (GSPT)} \cite{Fenichel1979,Jones1995,Wechselberger2019}. The main analytical obstacle to be overcome from the viewpoint of GSPT, stems from the presence of the singular exponential nonlinearity in \eqref{eq:hallo} as $ \eps= \beta^{-1}\to 0$. We note that \eqref{eq:hallo}, which is representative of problems with exponential nonlinearities quite generally, is smooth for each $\eps > 0$, but non-smooth in the limit $\eps \to 0$. This distinguishes it from the problems considered in e.g.~\cite{Dumortier2010,Miao2024,Popovic2007,Popovic2011,Popovic2011b,Popovic2012}, where travelling waves that are identified using GSPT and geometric blow-up feature a non-smooth cutoff which renders the equations non-smooth for $\eps > 0$.}

The travelling wave problem can be formulated as a boundary value problem in a planar system of singularly perturbed ODEs. \KUK{We use the geometric blow-up method to glue the convective-diffusive and reactive-diffusive zones together}, and subsequently identify a heteroclinic orbit which corresponds to a travelling wave in the original ZFK equation. Specifically, we show that for each $\eps > 0$ sufficiently small there exists a $(c,\eps)$-family of travelling waves, where $c > 0$ is the wave speed, and we prove the existence of a minimal wave speed $c = \overline c(\eps)$. 
These results can be seen as a rigorous and geometrically informative counterpart to existing results {that have been derived for similar problems using formal methods based on} HAEA. 
In addition to providing a rigorous justification of existing results, we \SJ{go beyond the existing results by (i) proving} smoothness of the minimum wave speed function $\overline c(\eps)$, and \SJ{providing} the asymptotics of $\overline c(\eps)$ up to and including $\mathcal O(\eps)$. This calculation shows that $\overline c(\epsilon)>1$ for all $0<\eps\ll 1$. Moreover, we (ii) identify and provide an asymptotic expansion for a slow manifold which characterises part of the wave profile for non-minimal wave speeds $c > \overline c(\eps)$ in the {convective}-diffusive zone. We use a normal form from \cite{DeMaesschalck2016} in order to derive the smoothness results in (i), and we use an approach inspired by \cite{Kristiansen2017} to derive the asymptotics in (ii). The approach used in (ii) is of independent interest, insofar as we expect that it can be applied in order to determine the asymptotics of flat slow manifolds quite generally. 



\PS{We expect our methods to apply to more complicated models for planar flame propagation away from so-called \textit{flammability limits} which -- like the simple ZFK equation considered herein -- can be described by a (system of) reaction-diffusion equation(s) 
for which the reaction rate is determined by an Arrhenius law with high activation energy (this is similar to the electrical oscillator models from \cite{Hester1968,LeCorbeiller1960} considered in \cite{Jelbart2019c}).} 

The manuscript is structured as follows: In Section \ref{sec:the_ZFK_equation}, we introduce the ZFK equation and \KUK{present a HAEA approach to the analysis of the ODE that describes the travelling wave problem, based upon {convective}-diffusive and reactive-diffusive zones}. \KUK{In Section \ref{sec:geometric_singular_perturbation_analysis}, we set up our geometric singular perturbation analysis of the problem of travelling waves}. {We state the main result in Section \ref{sec:main_result}, which is} proven using geometric blow-up and normal forms in Section \ref{sec:proof}. We conclude in Section \ref{sec:summary_and_outlook} with a summary and outlook.

\section{\KUK{The ZFK equation and HAEA}}
\label{sec:the_ZFK_equation}

Our starting point is the scalar ZFK equation
\begin{equation}
	\label{eq:ZFK_eqn}
	\theta_t = 
	\theta_{xx} + \omega(\theta, \beta) ,
\end{equation}
where $t \geq 0$, $x \in \R$ {and $\theta \in [0,1]$ is the so-called \textit{reduced temperature}, a non-dimensional variable which measures the \SJ{temperature} ratio of burnt to unburnt gas \SJ{in an Arrhenius reaction of the form
		\[
		R \to P + Q ,
		\]
		where $R$, $P$ and $Q$ denote the reactive premix/fuel species, the combustion products and the net energy release per mole of $R$ consumed respectively.}\footnote{\SJ{More precisely, $\theta(x,\cdot) = (T(x,\cdot) - T_u) / (T_b - T_u)$, where $T$, $T_u$ and $T_b$ denote the temperature of the gas, the completely fresh/unburnt and burnt gas respectively.}} 
	In particular, $\theta = 0$ corresponds to completely unburnt premix/fuel, and $\theta = 1$ corresponds to a completely burnt state.} {The function} $\omega(\theta,\beta)$ models the reaction term, which depends upon the \textit{Zeldovich number} $\beta > 0$. 
Equation \eqref{eq:ZFK_eqn} can be obtained via a reduction from a two component system of reaction-diffusion equations which models the propagation of planar flames if the diffusivity of the \SJ{``deficient reactant" (the reduced mass fraction associated with the reacting premix/fuel species $R$) is equal to the thermal diffusivity associated with the reduced temperature $\theta$,} 
i.e.~when the so-called \textit{Lewis number} $L \equiv 1$. We refer to \cite{Buckmaster1982,Bush1979,Clavin2016,Williams1985,Zeldovich1938} and the many references therein for details on the physical interpretation and derivation of the model. We consider the case for which the reaction term is given by
\begin{equation}
	\label{eq:omega}
	\omega(\theta, \beta) = \lambda \theta (1 - \theta) \e^{- \beta (1- \theta)} .
\end{equation}
{where $\lambda > 0$ is a parameter.} 
The asymptotic regime $0<\beta \ll 1$ is known as the \textit{KPP regime}, since equation \eqref{eq:ZFK_eqn} reduces to the KPP equation when 
$\beta \to 0$. We shall be interested in the opposite regime $\beta \gg 1$, known as the \textit{ZFK regime}, and the corresponding high activation energy limit $\beta \to \infty$. {The parameter $\lambda > 0$ can be scaled out using the parabolic rescaling $(x,t) \mapsto (x / \sqrt{\lambda}, t / \lambda )$, however, it is customary to fix a particular choice for $\lambda$ which ensures that the minimum wave speed derived in later sections is of order unity. We follow this convention here, and set} 
\[
\lambda = \frac{\beta^2}{2} :\quad \omega(\theta, \beta)=\frac{\beta^2}{2}\theta (1 - \theta) \e^{- \beta (1- \theta)} ,
\]
{without loss of generality.} This particular choice of $\omega$ and $\lambda$ was used in \cite[Ch.~8.3]{Clavin2016} in order to study the so-called \textit{KPP-ZFK transition} from \cite[Part.~F]{Velarde2012}. 
We shall write 
\[
\eps := \frac{1}{\beta} :\quad \omega(\theta, \epsilon^{-1})=\frac{1}{2\epsilon^2}\theta (1 - \theta) \e^{- \epsilon^{-1} (1- \theta)}
\]
and consider the asymptotic behaviour as $\eps \to 0$ (which corresponds to $\beta \to \infty$).
\begin{remark}\label{rem:this} {We restrict attention to the {physically} meaningful domain $\theta\in [0,1]$, but do note that $\omega$ is unbounded for $\theta>1$ for $\eps\to 0$. The graph of $\theta\mapsto \omega(\theta,\eps^{-1})$, $\theta\in [0,1]$, for five different values of $\eps$ is shown in Fig.~\ref{fig:graph}. Notice that $\omega(1,\eps^{-1})=0$ for all $\eps>0$ and that $\omega$ is exponentially small with respect to $\eps\to 0$ for fixed $\theta\in [0,1)$. Nevertheless, simple calculations show that $\omega$ is unbounded on $\theta\in [0,1]$ due a local maximum $\theta=\theta_c(\epsilon) =1+\mathcal O(\epsilon)$ with $\omega(\theta_c(\epsilon),\epsilon)=\mathcal O(\epsilon^{-1})$.  As we shall see, this lack of uniformity with respect to $\eps\to 0$ leads to distinct limiting problems in two distinct regimes.}
\end{remark}


\begin{figure}[t!]
\centering
\includegraphics[scale=0.5]{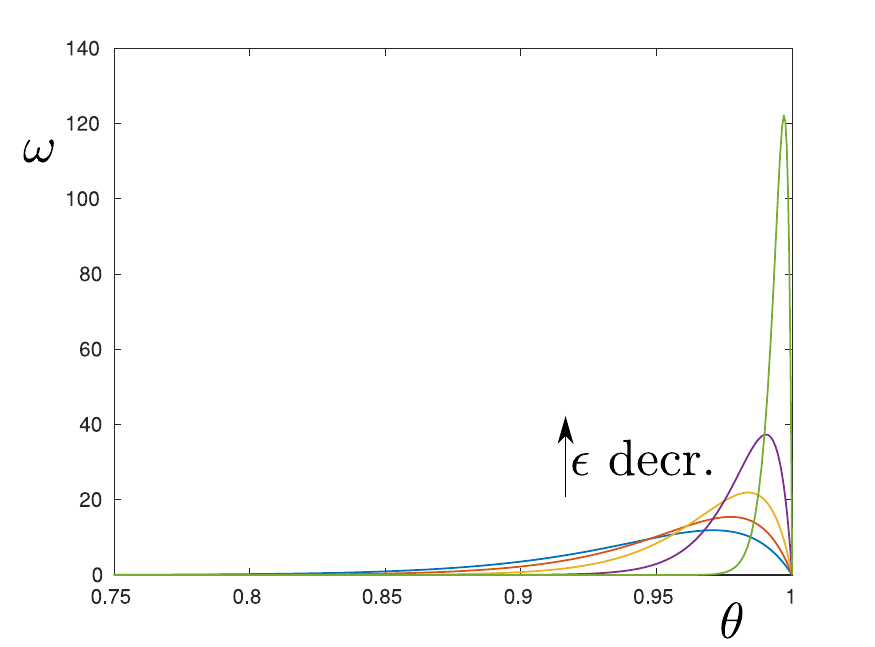}
\caption{Graph of $\theta\mapsto \omega(\theta,\eps^{-1})$ for five different values of $\eps$, equidistributed between $\eps=0.03$ (blue) and $\eps=0.003$ (green).}
\label{fig:graph}
\end{figure}

\begin{remark}
\label{rem:cold_boundary}
In combustion theory, one typically encounters the form
\begin{equation}
	\label{eq:omega_physical}
	\omega(\theta, \beta) = \lambda (1 - \theta) \e^{- \beta (1 - \theta)} 
\end{equation}
instead of \eqref{eq:omega}, see e.g.~\cite[Ch.~2.1 \& Ch.~8.2]{Clavin2016} and \cite{Buckmaster1982,Bush1979,Williams1985}. This form can be derived from the chemistry, however, the travelling wave problem is ill-posed because $\omega(0, \beta) =  \lambda \e^{- \beta} > 0$ for all $\beta > 0$, which means that the boundary condition at $\theta = 0$ (see equation \eqref{eq:bc} below) cannot be satisfied. This is known as the \textit{cold-boundary difficulty} \cite{Clavin2016,Williams1985}. In order to circumvent it, one can either (i) adjust the reaction term so that $\omega(0,\beta) = 0$, as in \eqref{eq:omega}, or (ii) introduce a cutoff by imposing that $\omega(\theta, \beta) \equiv 0$ for all $\theta < \theta_c$, where $\theta_c \in (0,1)$ is close to zero. The former approach is typically preferred in analytical studies such as ours, while the latter tends to be preferable for applications and simulations. 
Analytically, we expect that the ZFK equation \eqref{eq:ZFK_eqn} with \eqref{eq:omega_physical} and a cutoff can be treated using a combination of our approach with the geometric methods developed for problems 
with cutoffs in \cite{Dumortier2010,Miao2024,Popovic2007,Popovic2011,Popovic2011b,Popovic2012}.
\end{remark}


\subsection{\SJ{Formal construction of travelling waves via HAEA}}\label{sec:HAEA}

\SJ{Travelling wave solutions for $0 < \eps \ll 1$ which lie within the physically meaningful regime 
	$\theta \in [0,1]$ can be formally constructed using HAEA. In the following we provide an overview of this construction. While the exposition closely parallels asymptotic treatment of scalar ZFK equations in e.g.~\cite{Clavin2016,Williams1985}, the overall approach and HAEA methodology can be traced back to seminal papers by Zeldovich and Frank-Kamenetskii in 1938 in \cite{Zeldovich1938}.}

\SJ{As usual, the first step is to rewrite} equation \eqref{eq:ZFK_eqn} in the travelling wave frame $z = x + c t$, where $c \in \R$ is the (presently unknown) wave speed. \SJ{This leads to} the following second order ODE for travelling waves $\theta(x,t)=\widetilde\theta(z)$:
\begin{equation}
	\label{eq:ZFK_steady_state_ODE}
	c \frac{\dd \widetilde \theta}{\dd z} = \frac{\dd^2 \widetilde\theta}{\dd z^2} + \frac{1}{2 \eps^2} \widetilde \theta (1 - \widetilde\theta) \e^{- \eps^{-1} (1 - \widetilde\theta)} .
\end{equation}
We drop the tildes henceforth.
As in \cite[Ch.~2.1.2 \& 8.2]{Clavin2016}, the boundary conditions are
\begin{equation}
	\label{eq:bc}
	\lim_{z \to -\infty} \theta(z) = 0 , \qquad 
	\lim_{z \to \infty} \theta(z) = 1.
\end{equation}

\SJ{The key observation is that, due to the form of the reaction term \eqref{eq:omega}, there are two important regimes to consider:
	\begin{itemize}
		\item A \textit{{convective}-diffusive zone} defined by $1 - \theta = \mathcal O(1)$, where the reaction term in \eqref{eq:ZFK_system} is 
		\PS{exponentially}  small;
		\item A \textit{reactive-diffusive zone} defined by $1 - \theta = \mathcal O(\eps)$, where the advection term $c \frac{\dd \theta}{\dd z} $ is small relative to the reaction and diffusion terms due to the spatial scale,
	\end{itemize}
	see \cite[Fig. 2.4]{Clavin2016}.
	The basic idea in HAEA is to treat the leading convective-diffusive and reactive-diffusive zones as separate asymptotic regimes. In the convective-diffusive zone, the reaction term can be neglected, which leads to the simpler equation}
\[
\SJ{c \frac{\dd \theta}{\dd z} = \frac{\dd^2 \theta}{\dd z^2} .}
\]
\KUK{If one fixes the location of the wave according to $\theta(0)=1$, then the solution $\theta_-$ satisfying $\lim_{z\rightarrow -\infty} \theta(z) = 0$ is given by $ \theta_-(z) = \e^{c z}$. Notice that 
	\begin{equation}\label{eq:this}
		\theta_-'(0)=c.
	\end{equation}
	
	Next, in order to deal with the boundary condition at $z\rightarrow \infty$, we need to move to the reactive-diffusive zone. 
	For this purpose, we introduce scaled variables  $\theta_2$ and $z_2$ by setting $\theta = 1 + \eps \theta_2$ and $z = \eps z_2$; here the subscript ``2" is chosen for consistency with our notation in later sections. Substituting this into \eqref{eq:ZFK_steady_state_ODE}, using 
	\begin{equation}\label{eq:this3}
		\frac{\dd\theta}{\dd z}=\frac{\dd\theta_2}{\dd z_2},\quad \frac{\dd^2\theta}{\dd z^2}=\epsilon^{-1} \frac{\dd^2\theta_2}{\dd z_2^2},
	\end{equation}
	and simplifying leads to
	\begin{equation}
		\label{eq:rd_zone_eqn}
		\frac{\dd^2 \theta_2}{\dd z_2^2} = \frac{1}{2} \theta_2 \e^{\theta_2},
	\end{equation}
	as $\eps \to 0$. 
	In accordance with the original boundary condition as $z \to \infty$, recall \eqref{eq:bc}, we are now looking for a solution $\theta_{2+}$ of \eqref{eq:rd_zone_eqn} satisfying
	\begin{equation}\label{eq:this2}
		\lim_{z_2\rightarrow \infty}\theta_2(z_2) = 0.
	\end{equation}
	Equation \eqref{eq:rd_zone_eqn} has a first integral:
	\[H(\theta_2,\theta_2') = \frac{ \left(\theta_2'\right)^2}{2} -\frac{\theta_2-1}{2}e^{\theta_2},\]
	so that $\frac{\dd}{\dd z_2} H(\theta_2,\theta_2')=0$ when $\theta_2=\theta_2(z_2)$ solves \eqref{eq:rd_zone_eqn}.
	We 
	notice that $\theta_{2+}$ corresponds to the level set 
	\begin{equation}\label{eq:this4}
		H(\theta_{2},\theta_{2}')=\frac12,
	\end{equation} 
	and we select $\theta_{2+}$ so that it corresponds to the branch of the level set where $\theta_{2}$ is increasing (see also the phase portrait in Fig. \ref{fig:K2_singular_limit}). Therefore by letting $z_2\rightarrow -\infty$, we obtain that
	\begin{equation*}
		\lim_{z_2\rightarrow -\infty} \theta_{2+}(z_2)=-\infty\quad \mbox{and}\quad \lim_{z_2\rightarrow -\infty} \theta_{2+}'(z_2)=1 ,
	\end{equation*}
	where we for the latter have used \eqref{eq:this4}. Consequently, in order to match up with $\theta_-(z)$ at \SJ{$z=0$}, see \eqref{eq:this} and the first equation in \eqref{eq:this3}, we conclude that $c=1$ in this case.}

\section{Geometric singular perturbation analysis}
\label{sec:geometric_singular_perturbation_analysis}

\KUK{In this section, we set up our geometric singular perturbation analysis of the travelling wave problem. For this purpose, we define $\eta := \dd \theta / \dd z$ and rewrite equation \eqref{eq:ZFK_steady_state_ODE}} as the first order system
\begin{equation}
\label{eq:ZFK_system}
\begin{split}
	\dot \theta &= \eta, \\
	\dot \eta &= c \eta - \frac{1}{2 \eps^2} \theta (1 - \theta) \e^{- \eps^{-1} (1 - \theta)} ,
\end{split}
\end{equation}
where the overdot denotes differentiation with respect to $z$. There are two equilibria:
\begin{equation}
\label{eq:equilibria}
p_- : (0,0) , \qquad p_+ : (1, 0) ,
\end{equation}
which exist for all $c \in \R$ and $\eps > 0$. Direct calculations show that $p_-$ is a proper unstable node for all $c>2\epsilon^{-2} \e^{-\epsilon^{-1}}$, {whereas $p_+$ is a saddle. The strong and weak eigenvalues at the node $p_-$, which we denote here by $\lambda_s$ and $\lambda_w$ respectively, satisfy} 
\begin{align*}
\lambda_s := \frac{c}{2} \left(1 +  \sqrt{1 - \frac{2}{c\epsilon^{2} \e^{\epsilon^{-1}}}} \right) >\lambda_w := \frac{c}{2} \left(1 -  \sqrt{1 - \frac{2}{c\epsilon^{2} \e^{\epsilon^{-1}}}}\right) ,
\end{align*}
{and the corresponding} strong and weak eigenvectors {$(1,v_s)^\transpose$} and {$(1,v_w)^\transpose$ converge to $(1,c)^\transpose$ and $(0,1)^\transpose$ respectively as $\eps \to 0$.}
%
On the other hand, $p_+$ 
lies right at the interface
\begin{equation}
\label{eq:Sigma}
\Sigma = \{ (1, \eta) : \eta \in \mathbb R \} 
\end{equation}
between $\theta<1$, where 
{the reaction term is exponentially small}, and $\theta>1$, where {the right-hand side of} \eqref{eq:ZFK_system} is undefined for $\eps\to 0$.

Our aim is to identify and characterise (generally $(c,\eps)$-dependent) heteroclinic connections between $p_-$ and $p_+$ as $\eps \to 0$, which correspond to travelling wave solutions in the ZFK equation \eqref{eq:ZFK_eqn}. {Note that while heteroclinic orbits that connect to $p_-$ tangentially to the weak eigendirection are robust, heteroclinic orbits which connects to $p_-$ tangentially to the strong eigendirection are not. This follows from the uniqueness of the strong \SJ{unstable} manifold, which will be clarified in Lemma \ref{lemma:strong} below. In the following, we also refer to the heteroclinic connections themselves as `weak' or `strong' in order to distinguish these two cases.} 


\begin{remark}
\label{rem:strong_weak_heteroclinic}
System \eqref{eq:ZFK_system} has a symmetry
\begin{equation}
\label{eq:symmetry}
(\eta, z, c) \leftrightarrow (- \eta, -z, -c) .
\end{equation}
In the following we restrict attention to positive wave speeds $c > 0$, keeping in mind that corresponding statements can be derived for negative wave speeds using \eqref{eq:symmetry}. 

\end{remark}

\SJ{Motivated by the formal HAEA approach presented in Section \ref{sec:HAEA}, we begin 
with independent analyses of the geometry and dynamics in the convective-diffusive and reactive-diffusive zones.}

\subsection{Equations in the {convective}-diffusive zone}
\label{sub:outer_equations}

Taking $\eps \to 0$ in system \eqref{eq:ZFK_system} with $\theta\in [0,1)$ leads to the limiting problem in the {convective}-diffusive zone:{
\begin{equation}
\label{eq:ZFK_layer_problem}
\begin{split}
\dot \theta &= \eta , \\
\dot \eta &= c \eta ,
\end{split}
\end{equation}
}which does not depend on $\omega$. System \eqref{eq:ZFK_layer_problem} has a $1$-dimensional critical manifold
\[
S := \left\{ (\theta, 0) : \theta \in [0,1) \right\} ,
\]
and the Jacobian along $S$ has a single non-trivial eigenvalue $\lambda = c > 0$. Thus, $S$ is normally hyperbolic and repelling. Note that $p_- \in S$, but $p_+ \notin S$ because $\theta<1$ in the {convective}-diffusive zone. 
The fast fibers are contained within 
lines of the form
\[
\eta(\theta) = c \theta + \eta(0) ,
\]
see Fig. \ref{fig:singular_limit}, and the strong \SJ{unstable} manifold of $p_-$ is described by the following. 

\begin{figure}[t!]
\centering
\includegraphics[scale=0.35]{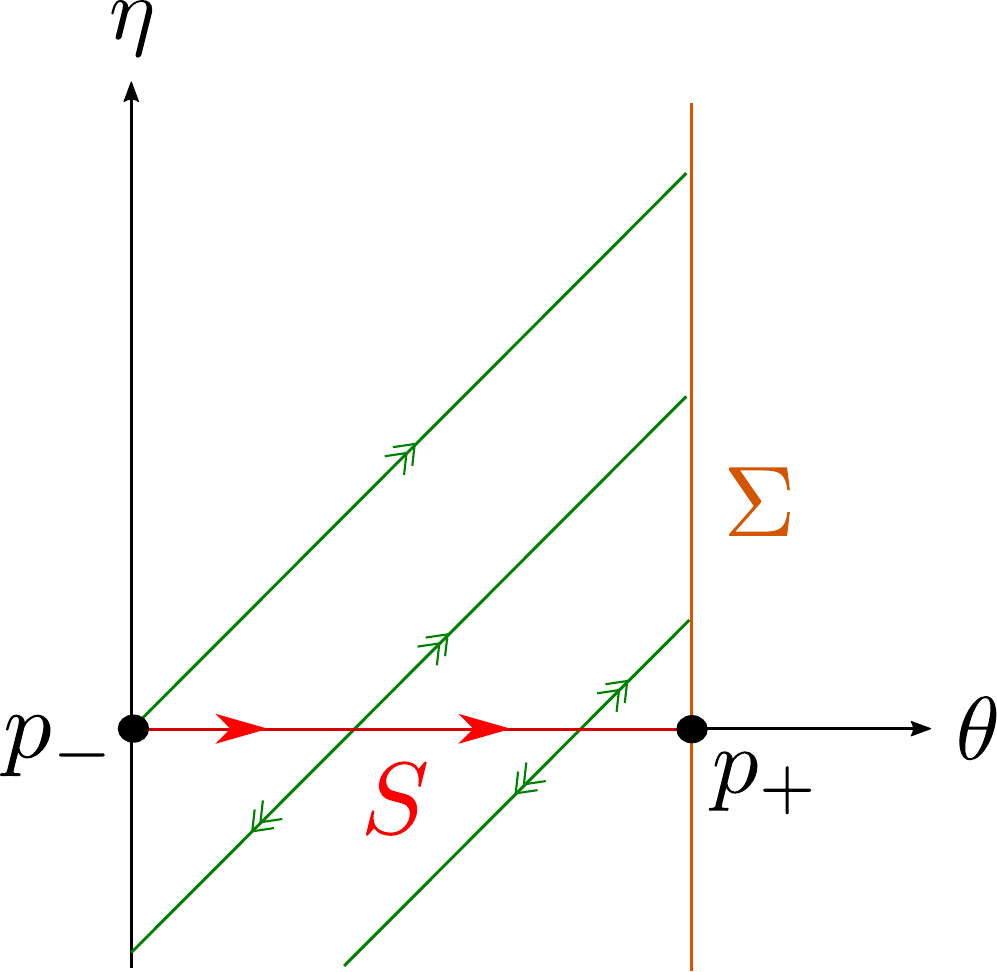}
\caption{Limiting geometry and dynamics of system \eqref{eq:ZFK_system} as $\eps \to 0$ in the {convective}-diffusive zone. 
We show the critical manifold $S$ (in red), and the set $\Sigma$ (in orange). 
$S$ is normally hyperbolic and repelling for $\theta \in [0,1)$. The fast fibers (in green) are contained in lines of slope $c$, with double arrows to indicate the `fast flow' (with respect to $z = x + ct$). The `slow flow' on $S$ is increasing between the equilibria $p_-$ and $p_+$, and indicated by single arrows. The colouring and labelling conventions adopted here are applied throughout.}
\label{fig:singular_limit}
\end{figure}

\begin{lemma}\label{lemma:strong}
	{Fix any $\kappa \in (0,1)$ and $c>0$. Then there exists an $\eps_0 = \eps_0(\kappa,c) > 0$ such that for all $\eps \in [0, \eps_0)$ the strong unstable manifold of $p_-$ is contained in a $C^\infty$-smooth graph defined by
			\begin{align}\label{eq:ss} \eta = c\theta + R(\theta,\eps),\end{align}
		over $\theta\in [0,1-\kappa]${, where} $R$ {is} infinitely flat with respect to $\epsilon\rightarrow 0^+$, i.e. 
		\begin{align*}
			\frac{\partial^n }{\partial \epsilon^n } R(\theta,0)=0\quad \forall\, n\in \mathbb N,
		\end{align*}
		such that $R(\theta,\epsilon)=\mathcal O(\eps^N)$ for all $N\in \mathbb N$.}
\end{lemma}

\begin{proof}
	{This follows directly from the Stable Manifold Theorem when $\eps \in (0,\eps_0)$, i.e.~when $p_-$ is a proper node \cite[Thm.~2.8]{Dumortier2006c}. However, it remains to prove the result for $\eps = 0$. Notice that Fenichel theory yields a strong unstable manifold given by a $C^N$-smooth graph of the form $\eta = c \theta + \mathcal O(\eps^N)$ for any $N \in \mathbb N$, however, it does not yield the $C^\infty$ property. In order to prove the $C^\infty$ property we apply} 
	%
	a polar blow-up of ${p_-} : (\theta,\eta)=(0,0)$. This turns the strong unstable manifold into a saddle and the result then follows from {\cite[Thm.~2.7]{Dumortier2006c}}. In the present case, we perform the directional blowup:
	\begin{align*}
		(r_1,\theta_1)\mapsto \begin{cases} \theta &= r_1\theta_1,\\ \eta&=r_1.\end{cases}
	\end{align*}
	Plugging this into \eqref{eq:ZFK_system} gives
	\begin{equation}\label{eq:r1theta1}
		\begin{aligned}
			\dot r_1 &= r_1\left(c - \frac{1}{2 \eps^2} \theta_1 (1 - r_1\theta_1) \e^{- \eps^{-1} (1 - r_1\theta_1)}\right) , \\
			\dot \theta_1 &= 1-\theta_1 \left(c - \frac{1}{2 \eps^2} \theta_1 (1 - r_1\theta_1) \e^{- \eps^{-1} (1 - r_1\theta_1)}\right).
		\end{aligned}
	\end{equation}
	Here $(r_1,\theta_1)=(0,c^{-1})$ is a hyperbolic saddle for $\eps=0$ with the stable manifold contained within the invariant set $r_1=0$. In fact, the system is $C^\infty$ smooth in a neighborhood of $(r_1,\theta_1)=(0,c^{-1})$, also with respect to $\eps\ge 0$. Consequently, there is a hyperbolic saddle for any $0\le \eps\ll 1$ with an unstable manifold that is $C^\infty$, also with respect to $c>0$ and $\eps\ge 0$. This invariant manifold corresponds to the strong unstable manifold upon blowing down. The expansion \eqref{eq:ss} is a consequence of \eqref{eq:r1theta1} being flat with respect to $\eps\rightarrow 0$. 
\end{proof}

In order to justify the orientation of the `reduced flow' indicated by single arrows on $S$ in Fig. \ref{fig:singular_limit}, note that Fenichel theory implies the existence of a locally invariant slow manifold of the form
\begin{equation}
	\label{eq:S_eps}
	S_\eps = \left\{ (\theta, h(\theta, \eps) ) : \theta \in [0, 1 - \kappa] \right\} ,
\end{equation}
for an arbitrarily small but fixed constant $\kappa \in (0,1)$. {{The function} $h$ is $C^N$-smooth for $\epsilon\in [0,\epsilon_0(N))$.} Determining the slow flow on $S_\eps$ is non-trivial, however, because the perturbation in system \eqref{eq:ZFK_system} is flat as $\eps \to 0$.

\begin{proposition}
	\label{prop:slow_manifold}
	The slow manifold $S_\eps$ has the following asymptotic expansion\SJ{:}
	\begin{equation}
		\label{eq:slow_manifold}
		\begin{split}
			\eta = h(\theta, \eps) &= \sum_{k=1}^\infty \frac{1}{2^k} F_k(\theta,\epsilon)  \epsilon^{-3k+1} \e^{-k\epsilon^{-1}(1-\theta)}\\
			&=\frac{1}{2c} \theta(1-\theta)  \epsilon^{-2} \e^{-\epsilon^{-1}(1-\theta)}+\mathcal O \left( \epsilon^{-5} \e^{-2\epsilon^{-1}(1-\theta)} \right),
		\end{split}
	\end{equation}
	with $F_1(\theta,\epsilon) = \frac
	{1}{c}\theta(1-\theta)$ and each $F_k$, $k\ge 2$, smooth satisfying {the recursion}
	\begin{align}\label{Fkrec}
		cF_k = \sum_{j=1}^{k-1} \left(\frac{\partial}{\partial \theta}F_j(\theta,\epsilon)\epsilon+jF_j(\theta,\epsilon) \right) F_{k-j}(\theta,\epsilon),\qquad k\ge 2.
	\end{align}
\end{proposition}

\begin{proof}
	As already noted, the slow manifold is flat with respect to $\epsilon\rightarrow 0$ and therefore we cannot determine the flow on the slow manifold in the usual way. Inspired by \cite{Kristiansen2017}, we proceed instead by looking for a way to rewrite \eqref{eq:ZFK_system} as an extended system in which the flat term $\frac12 \epsilon^{-2}\theta(1-\theta) \e^{-\epsilon^{-1}(1-\theta)}$ (or some scalings thereof) defines a new variable. In the present case, it turns out to be useful to  introduce $\tilde \eta$ and $\zeta$ as follows: 
	\begin{align}\label{eta2}
		\tilde \eta=\epsilon^{-2} \eta,\quad \zeta=\frac12  \epsilon^{-4} \e^{-\epsilon^{-1}(1-\theta)}. 
	\end{align}
	In this way, we can eliminate the flat terms by working in the extended $(\theta,\tilde \eta,\zeta)-$space where the equations are given by
	\begin{equation}\label{this2}
		\begin{aligned}
			\dot \theta &= \epsilon^2 \tilde \eta,\\
			\dot {\tilde \eta} &=c\tilde \eta-\theta(1-\theta) \zeta,\\
			\dot \zeta & = \epsilon \tilde \eta \zeta.
		\end{aligned}
	\end{equation}
	Notice that the set
	\begin{align}\label{invset}
		\mathcal Q := \left\{ (\theta,\tilde \eta,\zeta)\,:\,\zeta = \frac12 \epsilon^{-4} \e^{-\epsilon^{-1}(1-\theta)} \right\},
	\end{align}
	defines an invariant set for all $\epsilon\ge 0$, upon which \eqref{this2} reduces to \eqref{eq:ZFK_system}. System \eqref{this2} has a critical manifold defined by $\tilde \eta = \frac{1}{c} \theta(1-\theta)\zeta$, which is normally hyperbolic and repelling. Fenichel theory implies the existence of a slow manifold for system \eqref{this2} of the form
	$\tilde \eta = \frac{1}{c} \theta(1-\theta)\zeta+\mathcal O(\epsilon)$ on compact subsets ($\theta \in [0, 1 - \kappa]$ as in \eqref{eq:S_eps}), for all $0<\epsilon\ll 1$. 
	
	At the same time, $\tilde \eta=\zeta=0$ is also invariant for \eqref{this2}, and the linearization about any point in this set gives a single nonzero eigenvalue $c$ for all $\epsilon\ge 0$. The center space at a point $(\theta_0,0,{0})$ is spanned by the vectors $(1,0,0)^T$, $(0,-\theta_0(1-\theta_0),c)^ T$. We therefore have a center manifold that is a graph $\tilde \eta=G(\theta,\zeta;\epsilon)$ over $(\theta,\zeta)$. In particular, $G$ can be expanded as a formal series in $\zeta$ having $(\theta,\epsilon)$-dependent coefficients:
	\begin{align*}
		\tilde \eta = \sum_{k=1}^\infty G_k(\theta,\epsilon) \zeta^k.
	\end{align*}
	We find that each $G_k$, $k\ge 2$, is uniquely determined by the recursion relation
	\begin{align*}
		cG_k = \epsilon\sum_{j=1}^{k-1} \left(\frac{\partial}{\partial \theta}G_j(\theta,\epsilon)\epsilon+jG_j(\theta,\epsilon) \right) G_{k-j}(\theta,\epsilon),\quad k\ge 2.
	\end{align*}
	We have $G_1(\theta,\epsilon)=\frac{1}{c} \theta(1-\theta)$, $G_k(\theta,0)\equiv 0$ for all $k\ge 2$, and the center manifold is therefore a slow manifold ({i.e.~\SJ{a locally invariant} set that is $\mathcal O(\epsilon)$-close to the critical manifold}). 
	
	Making the ansatz $G_k=F_k \epsilon^{k-1}$, $k\ge 1$, 
	we find that $\epsilon^{k-1}$ cancels on both sides of the recursion relation and obtain the following equation for $F_k=F_k(\theta,\epsilon)$:
	\begin{align}\nonumber
		cF_k = \sum_{j=1}^{k-1} \left(\frac{\partial}{\partial \theta}F_j(\theta,\epsilon)\epsilon+jF_j(\theta,\epsilon) \right) F_{k-j}(\theta,\epsilon),\quad k\ge 2.
	\end{align}
	{This leads to the following expansion of the slow manifold of \eqref{this2}}:
	\begin{align}\label{eq:expansionG}
		\tilde \eta=G(\theta,\zeta;\epsilon)=\sum_{k=1}^\infty F_k(\theta;\epsilon)\epsilon^{k-1} \zeta^k.
	\end{align}
	We then take the intersection of \eqref{eq:expansionG}  {with the invariant set $Q$, recall \eqref{invset}}, and project the result back onto the $(\theta,\eta)$-plane using \eqref{eta2}. This leads to a slow manifold of \eqref{eq:ZFK_system} of the form
	\[
	\eta=\epsilon^2 G\left(\theta,\frac12 \epsilon^{-4} \e^{-\epsilon^{-1}(1-\theta)};\epsilon\right),
	\]
	which together with the expansion in \eqref{eq:expansionG} completes the result.
\end{proof}

Restricting system \eqref{eq:ZFK_system} to $S_\eps$, we find that
\[
\dot \theta \big|_{S_\eps} = \frac{1}{2 c \eps^2} \theta (1 - \theta) \e^{- \eps^{-1} (1 - \theta)} + \mathcal O \left( \eps^{-5} \e^{- 2 \eps^{-1} \theta (1 - \theta)} \right) 
\]
as $\eps \to 0$. {This shows that $\theta$ is increasing between $0$ and $1-\kappa\in (0,1)$ for all $\epsilon\in (0,\epsilon_0(\kappa))$ (i.e.~between $p_-$ and $p_+$ as $\eps\to 0$), as indicated in Fig. \ref{fig:singular_limit}.}

\begin{remark}
	The critical manifold $S$ can be extended up to $\theta = 1$, i.e.~up to $p_+$, but Fenichel theory cannot be applied near $p_+$ because the right-hand side of system \eqref{eq:ZFK_system} {is not a $C^1$-perturbation on $\theta\in [0,1]$ (only on a compact subset, see Fig. \ref{fig:graph})}. This is also evident from \eqref{eq:slow_manifold}, which shows that the validity of the series expansions for $S_\eps$ breaks down near $\theta = 1$.
\end{remark}

\begin{remark}
	A more ad-hoc way of arriving at the leading order asymptotics in the second line of \eqref{eq:slow_manifold} is to substitute the ansatz
	\begin{equation}
		\label{eq:ansatz}
		h(\theta, \eps) = \frac{1}{2 c \eps^2} \theta (1 - \theta) \e^{\eps^{-1} (1 - \theta)} + \tilde h(\theta, \eps)
	\end{equation}
	into \eqref{eq:ZFK_system}, and then use formal matching of terms to arrive at
	\[
	\tilde h(\theta, \eps) = \mathcal O \left( \eps^{-5} \e^{-2 \eps^{-1} (1 - \theta)} \right).
	\]
	We present a more detailed statement in Proposition \ref{prop:slow_manifold} in order to demonstrate the utility of a more general approach which is (i) of independent interest as a dynamical systems approach to the derivation of asymptotic series for flat slow manifolds, and (ii) applicable in situations where the ansatz \eqref{eq:ansatz} is not `obvious'.
\end{remark}

\subsection{Equations in the reactive-diffusive zone}
\label{sub:inner_equations}

In order to understand the dynamics in the reactive-diffusive zone, i.e.~close to $\Sigma$, we introduce the following coordinate translation and rescaling:
\[
\theta = 1 + \eps \theta_2.
\]
Inserting this into system \eqref{eq:ZFK_system} leads to
\begin{equation}
	\label{eq:K2_eqns}
	\begin{split}
		\theta_2' &= \eta , \\
		\eta' &= \frac{1}{2} \theta_2 \e^{\theta_2} + \eps \left( c \eta + \frac{1}{2} \theta_2^2 \e^{\theta_2} \right) ,
	\end{split}
\end{equation}
where the \SJ{prime} now denotes differentiation with respect to $z_2 = z / \eps$, and the subscript \SJ{``}2\SJ{"} is chosen in order for consistency with the usual notation for dynamics in the `rescaling chart' in geometric blow-up analyses (our primary tool in Section \ref{sec:proof}).

\begin{figure}[t!]
\centering
\includegraphics[scale=0.9]{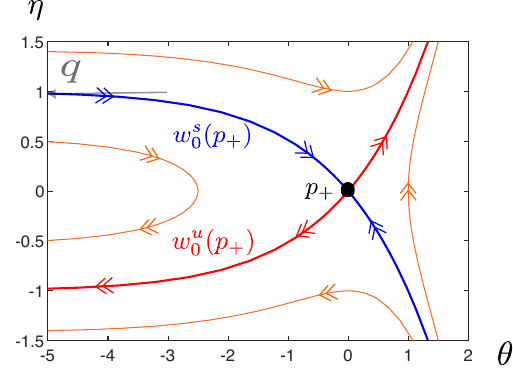}
\caption{Phase portrait for the limiting system \eqref{eq:K2_eqns_limiting} in the reactive-diffusive zone. The point $p_+$ is a resonant saddle with global stable/unstable manifolds $w_0^{s/u}(p_+)$ (shown in blue/red) given by \eqref{eq:stable_manifold}. We use a left-pointing arrow to indicate that the stable manifold \SJ{${w_0^s}(p_+)$} is asymptotic to the point $q$ as $\theta_2 \to \infty$.}
\label{fig:K2_singular_limit}
\end{figure}

System \eqref{eq:K2_eqns} is a regular perturbation problem, and the limiting system
\begin{equation}
\label{eq:K2_eqns_limiting}
\begin{split}
\theta_2' &= \eta , \\
\eta' &= \frac{1}{2} \theta_2 \e^{\theta_2} ,
\end{split}
\end{equation}
is Hamiltonian with solutions contained in level sets $H(\theta_2, \eta) = const.$, where
\begin{equation}
\label{eq:Hamiltonian}
H(\theta_2, \eta) =  \frac{\eta^2}{2} - \frac{\theta_2 - 1}{2} \e^{\theta_2} .
\end{equation}
System \eqref{eq:K2_eqns_limiting} has a hyperbolic saddle at $p_+ : (0,0)$, with eigenvalues $\lambda_\pm = \pm 1 / \sqrt 2$. Note that we permit a slight abuse of notation here; earlier we defined $p_+$ in \eqref{eq:equilibria} in $(\theta, \eta)$-coordinates. Using $H(0,0) = 1/2$ we find that the the global stable and unstable manifolds of $p_+$ are given by
$w_{0}^{s/u}(p_+) = \{ (\theta_2, h^{s/u}(\theta_2) : \theta_2 \in \R \}$, where
\begin{equation}
\label{eq:stable_manifold}
h^s(\theta_2) = -\operatorname{sign}(\theta_2) \sqrt{1 + (\theta_2 - 1) \e^{\theta_2} } , \quad
h^u(\theta_2) = \operatorname{sign}(\theta_2) \sqrt{1 + (\theta_2 - 1) \e^{\theta_2} } .
\end{equation}
Notice that $h^s(\theta_2) \to 1$ as $\theta_2 \to -\infty$. This suggests that we look for solutions in the {convective}-diffusive zone which connect to the point $q := (0,1) \in \Sigma$. The geometry and dynamics described above are sketched in Fig. \ref{fig:K2_singular_limit}.

\subsection{Singular heteroclinic orbits}

Using the above, we may construct an entire family of singular heteroclinic orbits. We define
\begin{equation}
\label{eq:sing_orbits}
\Gamma(c) := \Gamma^0(c) \cup \Gamma^1(c) \cup \Gamma^2 , \qquad c \in [1, \infty) ,
\end{equation}
where
\begin{equation*}
\begin{split}
\Gamma^0(c) &:= \left\{ (\theta, 0) : \theta \in \left[ 0, 1 - c^{-1} \right] \right\} , \\
\Gamma^1(c) &:= \left\{ (\theta, c \theta + 1 - c) : \theta \in \left[1 - c^{-1}, 1 \right) \right\} , \\
\Gamma^2 &:= \left\{ (0, \eta) : \eta \in [0, 1] \right\} .
\end{split}
\end{equation*}
Notice that {we only consider singular orbits $\Gamma(c)$ with $c \geq 1$, since connections for $c<1$ enter the nonphysical domain $\theta<0$.} In this sense, there is a minimum wave speed ($c=1$) for $\eps=0$. {The singular orbit $\Gamma(1)$, corresponding to the minimum wave speed, is distinguished by the fact that $\Gamma^0(1) = p_-$ and corresponds to a strong connection{; recall the discussion prior to Remark \ref{rem:strong_weak_heteroclinic}}.} In this case, there are only two components:
\[
\Gamma(1) = \Gamma^1(1) \cap \Gamma^2 .
\]
We sketch singular orbits with $c = 1$ and $c > 1$ in Fig. \ref{fig:singular_orbits}. The connections $\Gamma(c)$ with $c>1$ correspond to weak connections, {with the additional segment $\Gamma^0(c)$ corresponding to a slow segment along the critical manifold $S$}. 

\begin{figure}[t!]
\centering 
\subfigure[$c = 1$]{\includegraphics[scale=0.3]{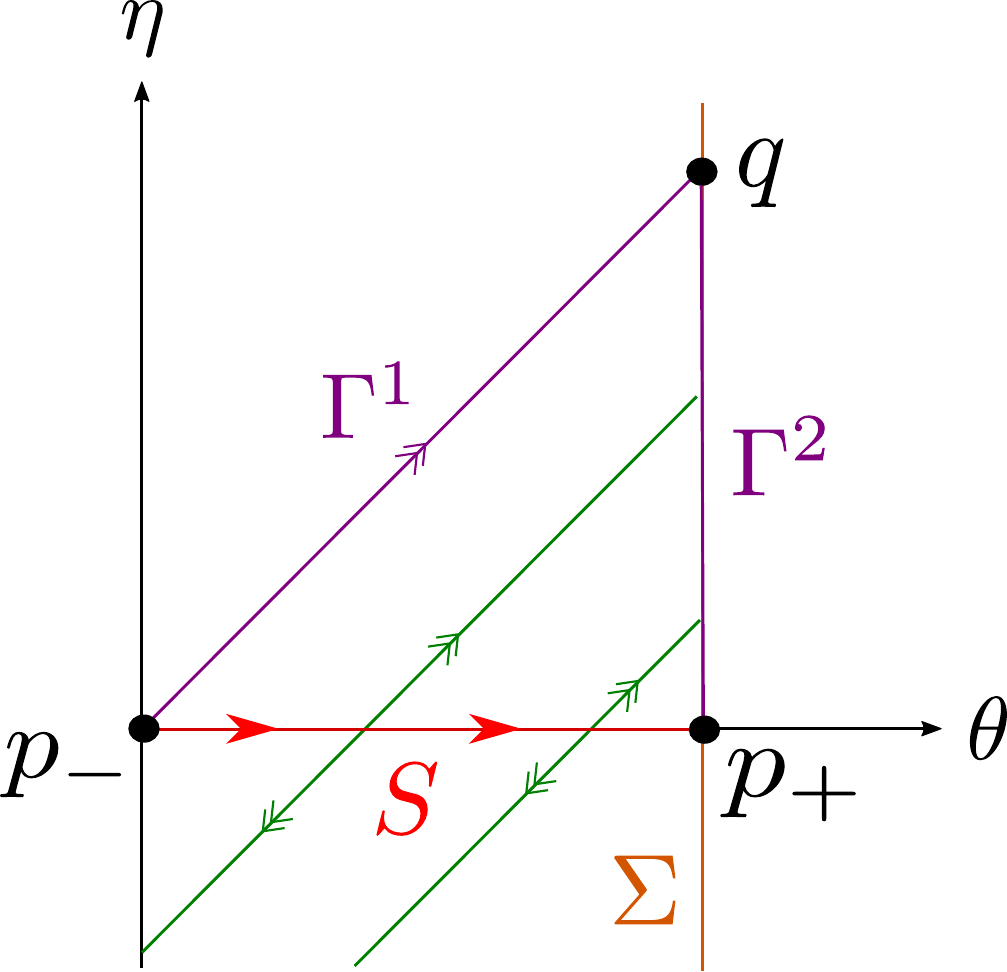}} \qquad
\subfigure[$ c > 1$]{\includegraphics[scale=0.3]{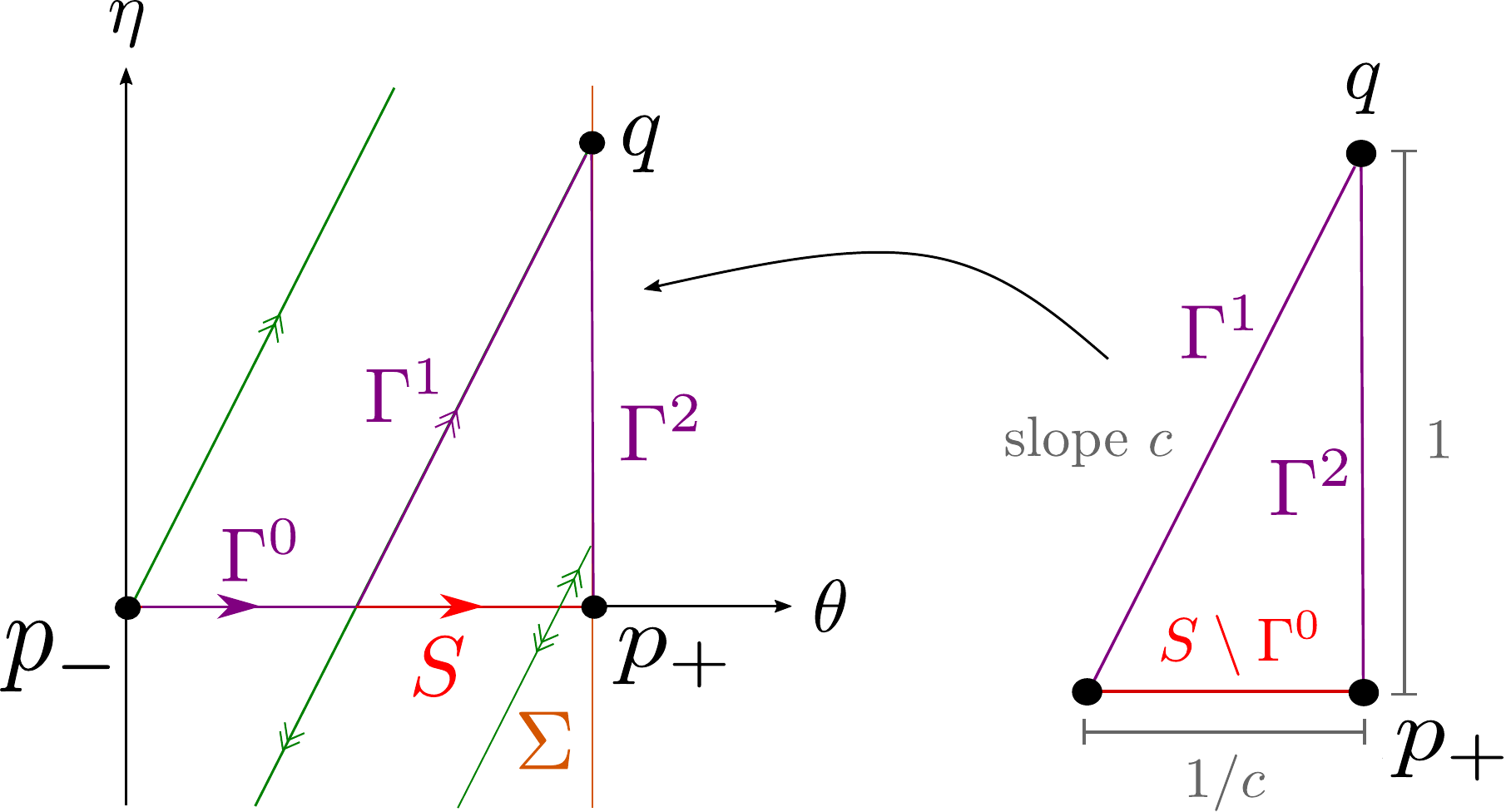}}
\caption{Two representative singular heteroclinic orbits $\Gamma(c)$ (with components shown in purple), as defined in \eqref{eq:sing_orbits}, for two different cases depending on the wave speed $c$. Left: Minimal wave speed $c = 1$, for which $\Gamma(1) = \Gamma^1(c) \cup \Gamma^2$ has only two distinct sections. Center: A fixed (non-minimal) wave speed $c > 1$, for which $\Gamma(c) = \Gamma^0(c) \cup \Gamma^1(c) \cup \Gamma^2$ has three distinct sections. Right: Schematic representation of the dimensions of the triangle with sides $\Gamma^1(c)$, $\Gamma^2$ and $S \setminus \Gamma^0(c)$, showing how $\Gamma^1(c)$ and $S \setminus \Gamma^0(c)$ (and therefore $\Gamma^0(c)$) change with $c$.}
\label{fig:singular_orbits}
\end{figure}

\section{{The main result}}
\label{sec:main_result}

We are now in a position to state the main result, namely, the existence of a family of heteroclinic solutions for all $\eps > 0$ sufficiently small in system \eqref{eq:ZFK_system} which converge to the (non-smooth) singular orbits $\Gamma(c)$ as $\eps \to 0$.

\begin{thm}
\label{thm:main}
Consider the ZFK system \eqref{eq:ZFK_system}. {Fix $\kappa>0$ small enough and $\sigma>0$ large enough}. Then there exists an $\eps_0 = \eps_0(\sigma) > 0$ and a {$C^\infty$-smooth function $\overline c : [0,\eps_0) \to \R$, satisfying 
%
\begin{equation}
\label{eq:wave_speed}
\overline c(0) = {1} ,\quad \overline c'(0) =  \int_{-\infty}^0 \left(1-h^s(x)\right) \dd x-1 \approx 0.34405>0,
\end{equation}
recall the definition of $h^s$ in \eqref{eq:stable_manifold},
such that the following assertions hold true for each $\epsilon \in (0,\eps_0)$}:
\begin{enumerate}
\item[(i)] There are no heteroclinic connections between $p_-$ and $p_+$ within {$\theta\in[0,1]$} for $c<\overline c(\eps)$.
\item[(ii)] There is a strong heteroclinic connection when $c = \overline c(\eps)$, whose 
corresponding orbit converges to $\Gamma(1)$ in Hausdorff distance as $\eps \to 0$.
\item[(iii)] There is a weak heteroclinic connection for each $c \in (\overline c(\eps), \sigma]$, and for each fixed $c\in (1+\kappa, \sigma]$, the corresponding orbit converges to $\Gamma(c)$ in Hausdorff distance as $\eps \to 0$.
\end{enumerate}
\end{thm}

The heteroclinic connection described by Theorem \ref{thm:main} corresponds to a travelling wave solution of the ZFK equation \eqref{eq:ZFK_eqn}. In essence, Theorem \ref{thm:main} provides a rigorous justification of formal asymptotic results that have been derived for similar problems using HAEA, most notably the existence of travelling wave solutions for each fixed $c$ greater than a minimum wave speed satisfying $c = \overline c(\eps) \sim 1$; we refer again to \cite{Clavin2016,Williams1985} and the references therein. Assertion (i) shows that there is minimum wave speed for $c = \overline c(\eps)$, and Assertions (ii) and (iii) describe the asymptotic properties of the wave profile as $\eps \to 0$. These two assertions describe an important qualitative difference between the cases $c = \overline c(\eps)$ and $c > \overline c(\eps)$, namely, that they correspond to profiles with two vs.~three distinct components as $\eps \to 0$ respectively; see again Fig. \ref{fig:singular_orbits}. 
The asymptotics of the `extra' component close to $\Gamma^0(c)$ when $c > \overline c(\eps)$ are described by the slow manifold asymptotics in Proposition \ref{prop:slow_manifold}. To the best of our knowledge, the smoothness of the minimum wave speed function $\overline c(\eps)$ has also not been proven. Finally, we emphasize that
\[
\overline c(\epsilon)>1\quad \forall\,0<\epsilon\ll 1,
\]
cf.~\eqref{eq:wave_speed}, \SJ{and} \KUK{that the HAEA approach in Section \ref{sec:HAEA} does not identify the weak connections}.

In Figs.~\ref{fig:num1} and \ref{fig:num2} we illustrate the results of numerical computations. The results are in \PS{good} agreement with Theorem  \ref{thm:main}. In particular, in Fig.~\ref{fig:num1} we show the stable manifold $W^s(p_+)$ (magenta) and the strong unstable manifold $W^{uu}(p_-)$ (green) for $\eps=0.01$ and two different $c$-values: $c=1.5$ in (a) and $c=\overline c(\epsilon)$ (using the linear approximation $\overline c(\epsilon) \approx 1+0.34405\times \epsilon$ provided by \eqref{eq:wave_speed}) in (b). The results were obtained using the linear approximations offered at $p_\pm$ and Matlab's ODE45 with low tolerances ($10^{-12}$). In (a), with $W^s(p_+)$ lying below $W^{uu}(p_-)$, we have a weak connection. The resulting profile is shown in Fig.~\ref{fig:num2} (dashed line). Due to exponentially slow flow along $S$, recall Proposition \ref{prop:slow_manifold}, the decay $\theta(z)\rightarrow 0$ for $z\rightarrow -\infty$ is slow and it is not visible in Fig.~\ref{fig:num2}(a). Fig.~\ref{fig:num1}(b) shows that the linear approximation offered by \eqref{eq:wave_speed} for $\overline c(\eps)$ is accurate for this value of $\eps$. In fact, $W^{uu}(p_-)$ and $W^s(p_+)$ \SJ{are} inseparable on the scale shown (being of the order $\sim 10^{-4}$). The resulting profile is shown in Fig.~\ref{fig:num2} (full line). See figure captions for further details.

\begin{figure}[t!]
\centering             
\subfigure[$c=1.5$]{\includegraphics[scale=0.5]{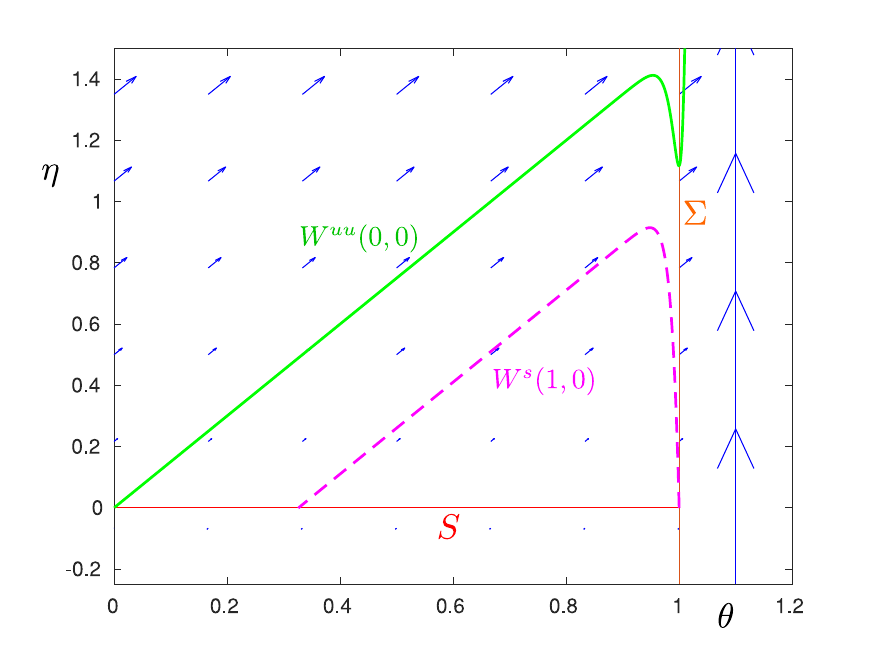}} \qquad
\subfigure[$c=\overline c(\epsilon)$]{\includegraphics[scale=0.505]{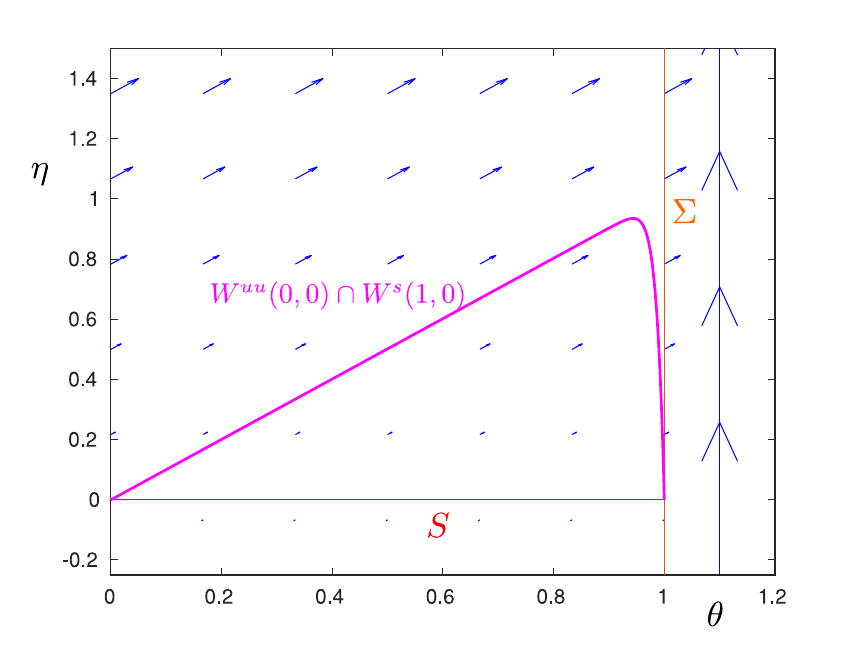}}
\caption{$W^s(1,0)$ for $c=1.5$ (in (a)) and $c=\overline c(\epsilon)$ (in (b)) for $\epsilon=0.01$. We have used the linear approximation $\overline c(\epsilon) \approx 1+0.34405\times \epsilon$ provided by \eqref{eq:wave_speed} in (b). The vector-field is illustrated in blue; notice that for illustrative purposes, the arrows for $\theta>1$ have been moderated.  }
\label{fig:num1}
\end{figure}
\begin{figure}[t!]
\centering             
\subfigure[$c=1.5$]{\includegraphics[scale=0.5]{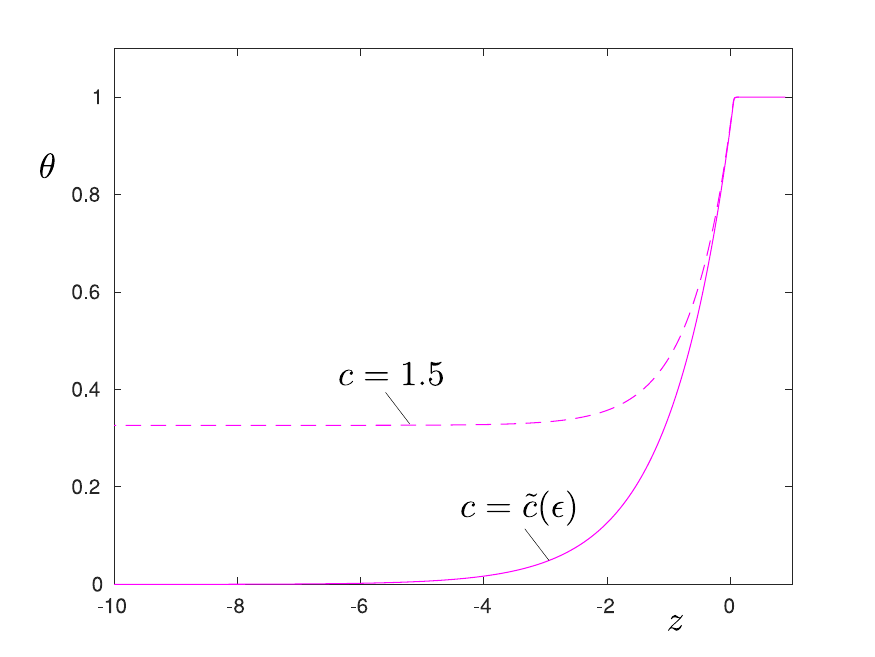}} \qquad
\subfigure[$c=\overline c(\epsilon)$]{\includegraphics[scale=0.5]{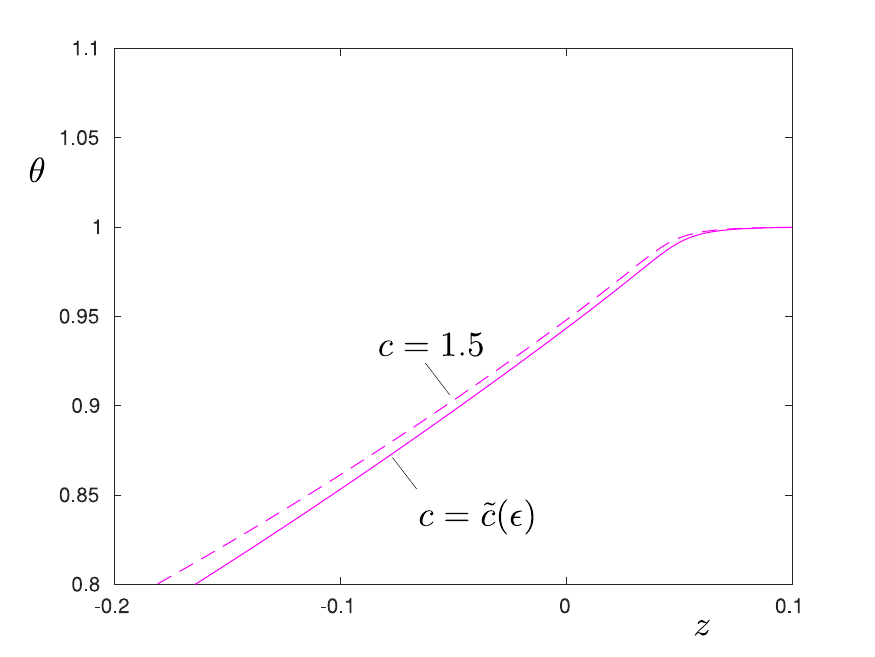}}
\caption{Profile $\theta(z)$ for $c=1.5$ (dashed) and $c=\overline c(\epsilon)$ (full line), with $\epsilon = 0.01$. (b) is a zoom of (a) around the sharp decay towards $\theta=1$. For $c=1.5$ the decay towards $\theta\rightarrow 0^+$ for $z\rightarrow -\infty$ is extremely slow.}
\label{fig:num2}
\end{figure}

\begin{remark}
\label{rem:c_selection}
The travelling wave corresponding to the minimum wave speed $c = \overline c(\eps)$ has received the most attention in the HAEA literature \cite{Clavin2016,Williams1985}, primarily because it is expected to be stable, see e.g.~\cite{Avery2022} for recent work on the so-called \textit{marginal stability conjecture}. We do not address questions of stability or selection in this work. 
\end{remark}


\section{\PS{Geometric blow-up and proof}}
\label{sec:proof}


As usual \PS{in approaches based on blow-up}, we begin by considering the extended system{
	\begin{equation}
		\label{eq:ZFK_system_extended}
		\begin{split}
			\dot \theta &= \eta, \\
			\dot \eta &= c \eta - \frac{1}{2 \eps^2} \theta (1 - \theta) \e^{- \eps^{-1} (1 - \theta)} , \\
			\dot \eps &= 0 ,
		\end{split}
	\end{equation}
}which is obtained from system \eqref{eq:ZFK_system} by appending the trivial equation {$\dot \eps = 0$}. 
The extended system \eqref{eq:ZFK_system_extended} is not defined {at the interface} $\Sigma = \left\{ (1, \eta, 0) : \eta \in \R \right\}$, {or more generally on the set defined by $\theta > 1$, $\eps = 0$.} 
	{Note that we permit a slight abuse of notation by using the same notation for $\Sigma$ as in \eqref{eq:Sigma} (the two sets are naturally identified). {In order to obtain a well-defined system on $\theta > 1$, $\eps = 0$, we} follow \cite{Jelbart2019c} and ``normalise the equations" through division of the right-hand side by 
		\begin{equation}\label{eq:quantity}
			\epsilon^{-2} \left(1+\frac12 \e^{\epsilon^{-1}(\theta-1)}\right). 
		\end{equation}
		This corresponds to a 
		\SJ{state-dependent transformation of the independent variable $z$} for $\eps>0$; we denote the new \SJ{independent variable} by $s$.
		In this way, we obtain an equivalent system for $\eps>0$, having a well-defined piecewise smooth limit $\eps\to 0$, with $\Sigma$ being the discontinuity set/{switching manifold}. In particular,
		\begin{equation}\label{eq:fuck}
			\begin{cases} \frac{\dd \theta}{\dd s}&=0,\\
				\frac{\dd \eta}{\dd s} &=\theta(\theta-1),
			\end{cases}\quad \mbox{for}\quad \theta>1,
		\end{equation}
		for $\eps\to 0$.
		In line with \cite{Kristiansen2019c}, we then \textit{gain smoothness} along $\Sigma$ by applying the cylindrical blow-up transformation:
		\begin{equation}
			\label{eq:blow-up}
			r \geq 0 , \ (\bar \theta, \bar \eps) \in S^1\cap \{\bar \eps\ge 0\} \ \mapsto \ 
			\begin{cases}
				\theta = 1 + r \bar \theta , \\
				\eps = r \bar \eps ,
			\end{cases}
		\end{equation}
		which fixes $\eta$. Notice specifically that 
		$\e^{-\epsilon^{-1}(1-\theta)} = \e^{-\bar \eps^{-1} \bar \theta}$ under \eqref{eq:blow-up}, which defines a smooth function on $(\bar \theta, \bar \eps)\in S^1$ with $\bar \epsilon\ge 0$ and $\bar \theta\le 0$.
	}
	
	{To analyse the normalised equations under the blowup transformation \eqref{eq:blow-up},} we work in local coordinate charts defined by $K_1 : \bar \theta = - 1$ and $K_2 : \bar \eps = 1$.  Local coordinates in charts are related to the original coordinates $(\theta, \eta, \eps)$ by
	\[
	K_1 : (\theta, \eta, \eps) = (1 - r_1, \eta, r_1 \eps_1) , \qquad
	K_2 : (\theta, \eta, \eps) = (1 + \eps \theta_2, \eta, \eps ) ,
	\]
	{which we insert into \eqref{eq:ZFK_system_extended} (and apply appropriate desingularization to ensure that the system is well-defined and smooth)}. The local change of coordinates formulae are given by
	\begin{equation}
		\label{eq:coord_changes}
		\kappa_{12} : \ r_1 = - \eps \theta_2 , \quad 
		\eps_1 = - 1 / \theta_2 , \quad \theta_2 < 0 , \qquad 
		\kappa_{21} : \ \theta_2 = - 1 / \eps_1 , \quad 
		\eps = r_1 \eps_1 , \quad \eps_1 > 0 .
	\end{equation}
	The coordinates in chart $K_2$ are just the coordinates of the inner equations used in Section \ref{sub:inner_equations}, except that we now view this system in the extended $(\theta_2, \eta, \eps)$-space. {(In this way, \eqref{eq:blow-up} also compactifies \eqref{sub:inner_equations}.)} {Recall that \eqref{sub:inner_equations} is a smooth regular perturbation problem on compact subsets.} Thus, it remains to consider the (smooth) equations in chart $K_1$, where the matching {between the inner and outer system} occurs.

	\subsection{Singular geometry and dynamics in $K_1$}
	
	The equations in chart $K_1$ are given by
	\begin{equation}
		\label{eq:K1_eqns_original}
		\begin{split}
			r_1' &= - r_1 \eta , \\
			\eta' &= c r_1 \eta - \frac{1}{2} \eps_1^{-2} (1 - r_1) \e^{- \eps_1^{-1}} , \\
			\eps_1' &= \eps_1 \eta ,
		\end{split}
	\end{equation}
	after a transformation of time which amounts to multiplication of the right-hand side by $r_1$. {We permit a slight abuse of notation by using a \SJ{prime} to indicate differentiation with respect to the new independent variable (as opposed to the usage in system \eqref{eq:K2_eqns}).} {Notice that system \eqref{eq:K1_eqns_original} is smooth on $r_1\ge 0$, $\epsilon_1\ge 0$, as desired}{, and that} the planes defined by $r_1 = 0$ and $\eps_1 = 0$, as well as their intersection; the line $r_1 = \eps_1 = 0$, are invariant. In particular, the set
	\[
	L := \{ (0, \eta, 0) : \eta \in \R \}
	\]
	defines a line of resonant saddles for system \eqref{eq:K1_eqns_original}, with eigenvalues $-\eta$, $0$ and $\eta$ (except for a non-hyperbolic point at $(0,0,0)$, which does not play a role in the analysis). 
	
	The limiting dynamics from the {convective}-diffusive zone appear within $\{ \eps_1 = 0 \}$, and are governed by
	\begin{equation*}
		\begin{split}
			r_1' &= - r_1 \eta, \\
			\eta' &= c r_1 \eta.
		\end{split}
	\end{equation*}
	{Indeed, this system is equivalent to \eqref{eq:ZFK_layer_problem} on $r_1>0$ upon setting $r_1=1-\theta$.
		We therefore also re-discover the repelling critical manifold $S$ as}
	\[
	S_1 = \{ (r_1, 0, 0) : r_1 \geq 0 \} .
	\]
	Considered within $\{ \eps_1 = 0\}$, the Jacobian evaluated along $S_1$ has eigenvalues $0$ and $c r_1 \geq 0$. Trajectories are contained in straight lines of the form $\eta(r_1) = - c r_1 + \eta(0)$.
	%
	%
	On the other hand, the limiting dynamics in the reactive-diffusive zone appear within $\{r_1 = 0\}$. We are particularly interested in the extension of {${w_0^s}(p_{+})$} into chart $K_1$. Using \eqref{eq:stable_manifold} and \eqref{eq:coord_changes}, we obtain
	\begin{equation}
		\label{eq:w_1}
		{w_{0,1}^s(p_+)} := \kappa_{21} \left( {w_0^s}(p_+) \times \{0\} \right) = 
		\{ (0, h_1^s(\eps_1), \eps_1) : \eps_1 \geq 0 \} ,
	\end{equation}                                                                                                where $h_1^s(\eps_1): = \sqrt{ 1 - \eps_1^{-1} (1 + \eps_1) \e^{- \eps_1^{-1}} }${, the subscript ``1" indicates that we are viewing the object in chart $K_1$,} and the solution within {$w_{0,1}^s(p_+)$} is backward asymptotic to the point $q : (0,1,0) \in L$. The singular geometry and dynamics is summarised in Fig. \ref{fig:blow-up}.

\begin{figure}[t!]
\centering                                
\subfigure[]{\includegraphics[scale=0.28]{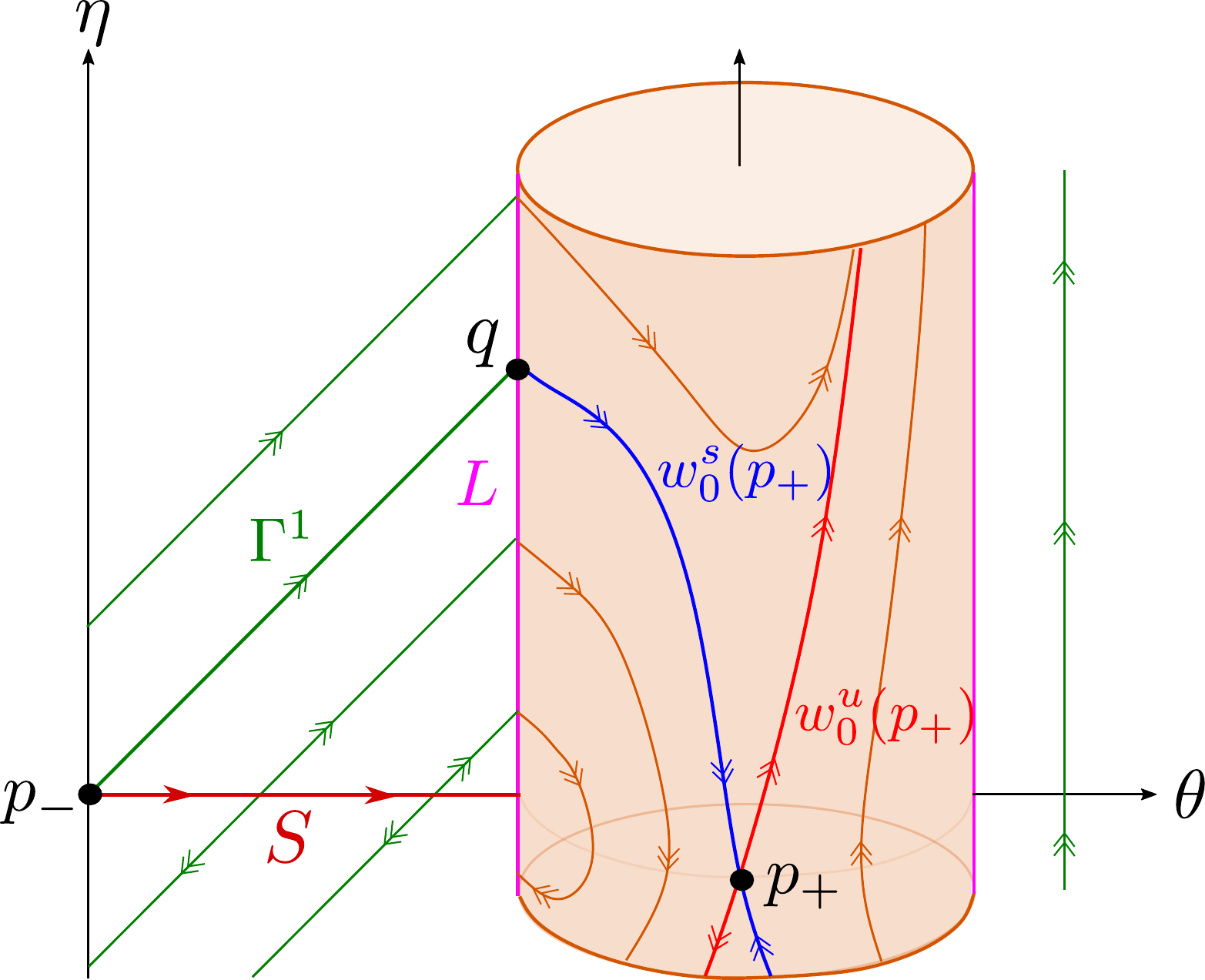}} \qquad 
\subfigure[]{\includegraphics[scale=0.28]{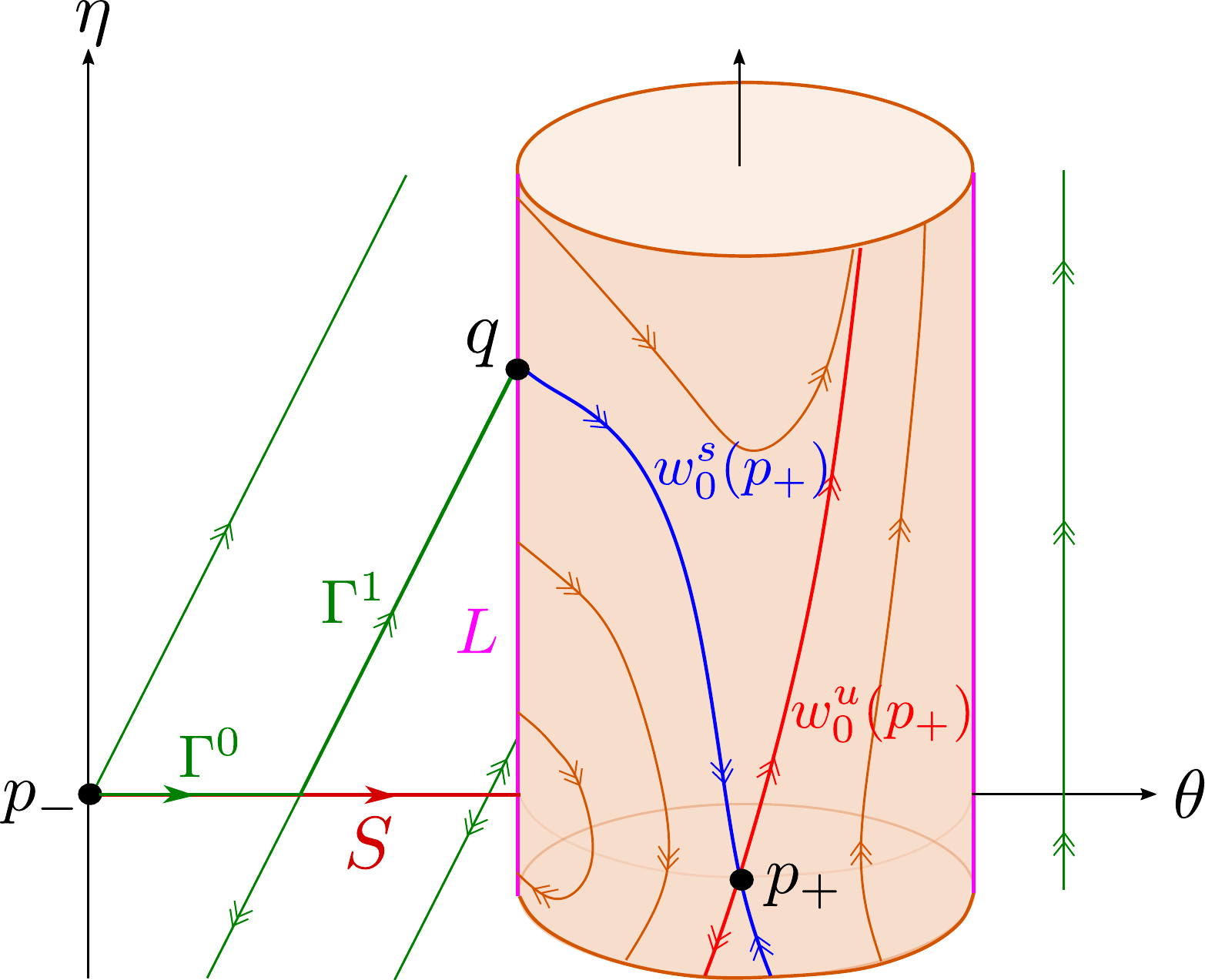}}
\caption{Singular geometry and dynamics after applying the blow-up transformation \eqref{eq:blow-up} and desingularising in order to resolve the loss of smoothness along $\Sigma$.  In contrast to earlier figures, we also indicate the dynamics for $\theta > 1$, which {is well-defined after the normalization of the system through division of \eqref{eq:ZFK_system_extended} by the quantity \eqref{eq:quantity}, see also \eqref{eq:fuck} and \cite{Jelbart2019c,Kristiansen2019c} for further details on this procedure.} The dynamics in the {convective}-diffusive and reactive-diffusive zones extend and `connect' along a line of hyperbolic saddles $L$, shown here in magenta, which contains the point $q$ identified as the asymptote of the stable manifold ${w_0^s}(p_+)$ in the Subsection \ref{sub:inner_equations}. As in Fig. \ref{fig:singular_orbits}, we show two different scenarios depending on whether $c = 1$ or $c > 1$. Left: $c = 1$, with the improved singular heteroclinic $\Gamma^1(1) \cup w_0^s(p_+)$. Right: $c > 1$, with the improved singular heteroclinic $\Gamma^0(c) \cup \Gamma^1(c) \cup w_0^s(p_+)$.}
\label{fig:blow-up}
\end{figure}

The preceding analysis allows for the construction of an improved family of singular heteroclinics
\begin{equation}
\label{eq:improved_singular_heteroclinics}
\Gamma(c) = \Gamma^0(c) \cup \Gamma^1(c) \cup \Gamma^2 ,
\end{equation}
considered now in the blown-up space. The definitions for $\Gamma^0(c)$ and $\Gamma^1(c)$ remain unchanged (after a natural embedding into the extended $(\theta, \eta, \eps)$-space or representation in chart $K_1$ coordinates), but $\Gamma^2$ is replaced by
\[
\Gamma^2 = q \cup {({w_0^s(p_+)} \times \{0\})}.
\]                                                                         
Improved singular heteroclinic orbits for $c = 1$ and $c > 1$ are shown in Fig. \ref{fig:blow-up}.

\subsection{Perturbation}

It remains to describe the perturbation of the improved singular heteroclinics in \eqref{eq:improved_singular_heteroclinics}. The main task is to understand the local passage near $q \in L$. To this end, we divide the right-hand side of system \eqref{eq:K1_eqns} by $\eta$ (which is positive close to $q$) and consider the flow induced by the system
\begin{equation}
\label{eq:K1_eqns}
\begin{split}
r_1' &= - r_1 , \\
\eta' &= c r_1 - \frac{1}{2} \eta^{-1} \eps_1^{-2} (1 - r_1) \e^{- \eps_1^{-1}} , \\
\eps_1' &= \eps_1 .
\end{split}
\end{equation}
System \eqref{eq:K1_eqns} can be brought into a kind of local normal form by straightening the fibers of the stable and unstable manifolds with base points along $L$, which are contained in the invariant planes $\{r_1 = 0\}$ and $\{\eps_1 = 0\}$ respectively. This can be achieved with an $(r_1,\epsilon_1)$-fibered diffeomorphism defined by
\begin{equation}
\label{this}
(r_1,y_1,\epsilon_1)\mapsto \eta=-cr_1 +f_1(y_1,\epsilon_1),
\end{equation}
where
\[
f_1(y_1,\epsilon_1) = y_1 \sqrt{1- y_1^{-2} \left(\eps_1^{-1} + 1\right)\e^{-\epsilon_1^{-1}}}
\]
and the inverse is given by
\begin{equation}
\label{eq:y1_inverse}
y_1 = \sqrt{(\eta + c r_1)^2 + \left(\eps_1^{-1} + 1 \right) \e^{-\epsilon_1^{-1}}} .
\end{equation}                                                                                                                                                                                              
Indeed, in these coordinates the sets defined by $(r_1,y_1,0),\,r_1\ge 0$ and $(0,y_1,\epsilon_1),\,\epsilon_1\ge 0$ with $y_1>0$ fixed are the stable and unstable manifolds of the point $(0,y_1,0)$; see Fig. \ref{fig:normal_form}. Here we have used the conservation of $H(-\eps_1^{-1},\eta)=\text{const.}$ within $r_1=0$ to write the fibers of the unstable manifold with base point $(r_1,y_1,\epsilon_1)=(0,y_1,0)$ as 
\begin{equation}
r_1=0,\quad H(-\epsilon_1^{-1}, \eta)=\frac{y_1^2}{2}. \label{y1def}
\end{equation}
Solving the last expression for $\eta$ gives $\eta=f_1(y_1,\epsilon_1)$. Notice that 
\begin{align} \label{f1prop}
f_1(1,\epsilon_1) = h^s(-\epsilon_1^{-1}),
\end{align} 
since $\eta = h^s(\theta_2)$ corresponds to the level set $H(\eta,\theta_2)=\frac12$ (i.e. $y_1=1$).
This leads to the following equations:
\begin{equation}\label{basecase}
\begin{aligned}
\dot r_1 &= -r_1,\\
\dot y_1 & = r_1  F_1(r_1,y_1,\epsilon_1,c), \\
\dot \epsilon_1& =\epsilon_1,
\end{aligned}
\end{equation}
where
\begin{equation}
\label{eq:F1}
F_1(r_1,y_1,\epsilon_1,c):=\frac12 y_1^{-1}  \epsilon_1^{-2} \e^{-\epsilon_1^{-1}} \frac{f_1(y_1,\epsilon_1)-c}{f_1(y_1,\epsilon_1)-cr_1 }.
\end{equation}
Notice that $F_1$ is \textit{infinitely flat} in $\epsilon_1\rightarrow 0$, i.e.~
\begin{equation}
\frac{\partial^n }{\partial \epsilon_1^n}F_1(r_1,y_1,0,c)=0,\label{infflat}
\end{equation}
for all $n\in \mathbb N$. We shall adopt the following notation: for any $N\in \mathbb N\cup \{\infty\}$, we let $j_N G(\cdot)$ denote the $N$th-order Taylor-jet of a $C^N$-smooth function $x\mapsto G(x)$ defined in a neighborhood of $x=0$. Then we can write \eqref{infflat} as $j_\infty F_1(r_1,y_1,\cdot,c) = 0$.

\begin{figure}[t!]
\centering                     
\includegraphics[scale=0.3]{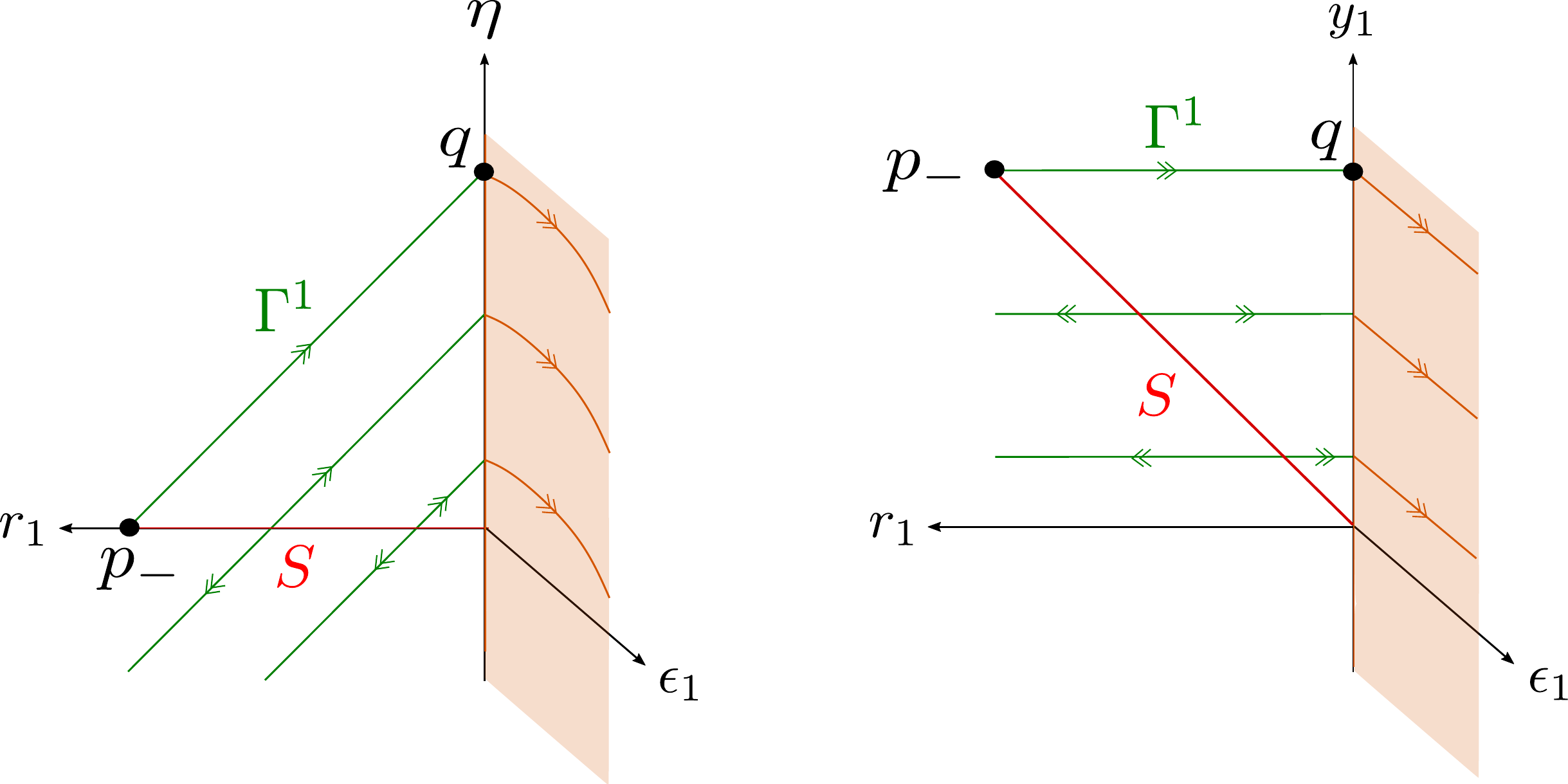}
\caption{Singular geometry and dynamics in chart $K_1$ before (left) and after (right) applying the normal form transformation \eqref{this}, which rectifies the stable and unstable manifolds within the invariant planes $\{ r_1 = 0\}$ and $\{\eps_1 = 0\}$ locally near $L$.}
\label{fig:normal_form}
\end{figure}

System \eqref{basecase} is similar (but not identical) to the normal form in \cite[Prop.~3.2]{DeMaesschalck2016} (here $F_1$, called $k$ in \cite[Eq.~(2.13)]{DeMaesschalck2016}, only depends upon $r_1$ through $r_1\epsilon_1=\epsilon$). {Let $\pi : \Delta_1^{in} \to \Delta_1^{out}$ denote the transition map} induced by the flow of system \eqref{basecase}, where
\begin{equation}
\label{eq:sections}
\begin{split}
\Delta^{in}_1 &= \left\{ (\rho, y_1, \eps_1) : y_1 \in I , \eps_1 \in [0, \delta] \right\} , \\
\Delta^{out}_1 &= \left\{ (r_1, y_1, \delta) : r_1 \in [0,\rho] , y_1 \in \tilde I \right\} ,
\end{split}
\end{equation}
with $I := [1-\beta, 1+\beta] \subset (0,\infty)$ and $\tilde I := [1 - \tilde \beta, 1 + \tilde \beta]$ for some fixed $\beta \in (0,1)$, $\tilde \beta > \beta$, and small but fixed $\rho, \delta > 0$. {Due to resonances along $L$, one would in general expect logarithmic terms in the expansion of $\pi$ (see e.g. \cite{DeMaesschalck2016}). However, all resonant terms are absent in the present case due to the flatness \eqref{infflat} with respect to $\epsilon_1$. Indeed, we have the following:}
\begin{proposition}
\label{prop:Y_smoothness}
Fix any $N\in \mathbb N$ and a compact interval $J\subset (0,\infty)$. Then for $\rho>0, \delta>0$ sufficiently small, there exists a $C^N$-smooth function $Y : I \times [0,\delta) \times J \rightarrow \mathbb R$ such that
the map $\pi : \Delta_1^{in} \to \Delta_1^{out}$ is well-defined and given by
\begin{equation}
\label{eq:pi}
\pi : (\rho, y_1, \eps_1) \mapsto 
\left( \frac{\rho}{\delta} \eps_1, Y(y_1, \eps_1, c ; \rho, \delta), \delta \right) .
\end{equation}
\end{proposition}

\begin{proof}
The first component of the right-hand side in \eqref{eq:pi} follows directly from $\eps = \eps_1 r_1 = \text{const.}$ It therefore remains to show that the function $Y$ is well-defined and $C^N$-smooth. In order to do so, we show that system \eqref{basecase} can be brought into the local normal form in \cite[Prop.~3.2]{DeMaesschalck2016}, which, based on {additional arguments} in \cite{DeMaesschalck2016}, is sufficient to prove $C^N$-smoothness of $\pi$. Since $\rho, \delta>0$ are fixed as parameters, we shall suppress the dependency of $Y$ on these quantities. We also suppress the dependency on $c$; the smoothness with respect to this parameter should be obvious from the proof.

First, we claim that for any $N\in \mathbb N$ there is an $(r_1,\epsilon_1)$-fibered $C^\infty$-diffeomorphism
\begin{align*}
y_1 = y_N+\alpha_N(r_1,y_N,\epsilon_1),
\end{align*}
with $j_\infty \alpha_N(r_1,y_N,\cdot)=0$, such that
\begin{equation}
\begin{split}
\label{eq:y_N_system}
\dot r_1 &= -r_1,\\
\dot y_N & = r_1^N F_N(r_1,y_N,\epsilon_1), \\
\dot \epsilon_1& =\epsilon_1,
\end{split}
\end{equation}
with $j_\infty F_N(r_1,y_N,\cdot)=0$. Following \cite{DeMaesschalck2016}, we proceed by induction. We have already shown the base case $N = 1$ in the derivation of system \eqref{basecase} using \eqref{this}. We therefore proceed directly to the induction step and suppose that the statement is true for $N=n$. We write
\begin{equation}
\label{eq:y_n}
y_{n+1} = y_n + r_1^n \beta_n(y_n,\epsilon_1),
\end{equation}
with the (as yet undetermined) $\beta_n$ satisfying $j_\infty \beta_n(y_n, \cdot) = 0$. Writing 
\[
F_n(r_1,y_n,\epsilon_1) = F_{n,0}(y_n,\epsilon_1) +r_1 F_{n,1}(r_1,y_n,\epsilon_1) ,
\]
we obtain 
\begin{equation}
\dot y_{n+1} = r_1^n \left\{F_{n,0}(y_{n+1},\epsilon_1) -n \beta_n(y_n,\epsilon_1) + \epsilon_1 \frac{\partial }{\partial \epsilon_1}\beta_n(y_n,\epsilon_1)\right\}+\mathcal O(r_1^{n+1}).\label{yn1eqns}
\end{equation}
The $\mathcal O(r_1^n)$ terms vanish if we impose the following solvability condition on $\beta_n$:
\begin{equation}\label{betaneqn}
\epsilon_1\frac{\partial }{\partial \epsilon_1}\beta_n(y_n,\epsilon_1) =n \beta_n(y_n,\epsilon_1)-F_{n,0}(y_n,\epsilon_1) ,\quad j_\infty \beta_n(y_n, \cdot) = 0,
\end{equation}
which may be solved to give
\begin{equation}
\label{eq:beta_n}
\beta_n(y_n,\eps_1) = - \eps_1 \int_0^{\eps_1} s^{-2} F_{n,0}(y_n,s) \dd s.
\end{equation}

Now fix $N+3\in \mathbb N$ and use $\eps = r_1 \eps_1$ to rewrite system \eqref{eq:y_N_system} as
\begin{equation}
\label{eq:normal_form}
\begin{split}
\dot r_1 &=-r_1,\\
\dot y_{N+3} &= \epsilon^{N+3} R(r_1, y_{N+3}, \epsilon_1), \\
\dot \epsilon_1 &=\epsilon_1,
\end{split}
\end{equation}
where $R:=\epsilon_1^{-(N+3)} F_{N+3}$. Notice that this is smooth since $j_\infty F_{N+3}(r_1,y_n,\cdot)=0$. System \eqref{eq:normal_form} is -- up to a change of notation -- precisely the local normal form appearing in \cite[Prop.~3.2]{DeMaesschalck2016} (case $a = 0$). $C^N$-smoothness of the function $Y$ then follows from the arguments in \cite[Sec.~3.2]{DeMaesschalck2016}.
\end{proof}

Using {Proposition} \ref{prop:Y_smoothness} and a number of elements of the proof thereof, we can derive the leading order asymptotics for $Y(y_1, \eps_1,c)$.

\begin{lemma}
\label{lem:Y_asymptotics}
The function $Y$ has the following asymptotic expansion in $\eps_1$:
\begin{equation}
\label{eq:Y_asymptotics}
Y(y_1, \eps_1, c) = 
y_1 + \eps_1 \rho \int_0^{\delta} s^{-2} F_1(0,y_1,s,c) \dd s + \mathcal O(\epsilon_1^2) ,		
\end{equation}
where the function $F_1$ is given in \eqref{eq:F1}.
\end{lemma}

\begin{proof}
From \eqref{this} and the proof of {Proposition} \ref{prop:Y_smoothness} we have the following relations:
\[
y_{N+3}^{out} = y_{N+3}^{in} + \mathcal O(\eps^N), \quad 
y_1 = y_{N+3} + \alpha_{N+3}(r_1, y_{N+3}, \eps_1), \quad
y_{N+3} = y_1 + \tilde \alpha_1(r_1, y_1, \eps_1) ,
\]
where $\tilde \alpha_1(r_1, y_1, \eps_1)$ is $C^N$-smooth and $\mathcal O(\eps_1)$. Using these relations to evaluate $y_1^{out} = Y(y_1, \eps_1, c)$, we obtain
\begin{equation}
\label{eq:Y}
Y(y_1, \eps_1, c) = y_1 + \hat \alpha_1(\rho, y_1, \eps_1) ,
\end{equation}
where the function $\hat \alpha_1(\rho, y_1, \eps_1)$ is $C^N$-smooth and $\mathcal O(\eps_1)$ as $\eps_1 \to 0$. In order to determine $\hat \alpha_1(\rho, y_1, \eps_1)$ to leading order in $\eps_1 \to 0$, we use the transformation $y_2 = y_1 + \alpha_1(r_1, y_1, \eps_1) = y_1 + r_1 \beta_1(y_1, \eps_1)$ and the fact that
\[
y_2^{out} = y_2^{in} + \hat \alpha_2(\rho, y_2^{in}, \eps_1^{in}) ,
\]
where the function $\hat \alpha_2(\rho, y_2^{in}, \eps_1^{in})$ is $C^{N-1}$-smooth and $\mathcal O((\eps_1^{in})^2)$ as $\eps_1^{in} \to 0$ (this follows from relations analogous to those which led to \eqref{eq:Y}). We obtain
\begin{equation*}
\begin{split}
Y(y_1, \eps_1, c) &= 
y_2^{out} - \frac{\rho}{\delta} \eps_1 \beta_1(Y(y_1, \eps_1, c), \delta) \\
&= y_2^{in} -  \frac{\rho}{\delta} \eps_1 \beta_1(y_1, \delta) + \mathcal O(\eps_1^2) \\
&= y_1 - \frac{\rho}{\delta} \eps_1 \beta_1(y_1, \delta) + \mathcal O(\eps_1^2) .
\end{split}
\end{equation*}
Substituting the expression for $\beta_1(y_1, \delta)$ in \eqref{eq:beta_n} leads to \eqref{eq:Y_asymptotics}, as required.
\end{proof}

Taken together, {Proposition} \ref{prop:Y_smoothness} and {Lemma} \ref{lem:Y_asymptotics} imply that the transition map $\pi : \Delta_1^{in} \to \Delta_1^{out}$ is $C^N$-smooth of the form \eqref{eq:pi}, with $Y(y_1,\eps_1,c)$ given by \eqref{eq:Y_asymptotics}. We now use Lemma \ref{lem:Y_asymptotics} in order to derive a bifurcation equation on the section $\Delta_1^{out}$. More precisely, we introduce a distance function on {$\Delta_1^{out}$} via
\begin{equation}
\label{eq:distance_fn}
\mathcal D(c,\eps) := {Y_u(c,\eps) - Y_s(c,\eps)} .
\end{equation}
{Here $Y_u(c,\eps)$} denotes the $y_1$-coordinate of the intersection ${w_{\eps,1}^u}(p_-) \cap \Delta_1^{out}$, where ${w_{\eps,1}^u}(p_-)$ denotes the (unique) strong unstable manifold of the point $p_-$ {in $K_1$ coordinates}, and {$Y_s(c,\eps)$} denotes the $y_1$-coordinate of the intersection ${w_{\eps,1}^s}(p_{+}) \cap \Delta_1^{out}$, where ${w_{\eps,1}^s}(p_{+})$ denotes the (unique) stable manifold of the saddle $p_{+}$, {which is a perturbation of the limiting stable manifold in \eqref{eq:w_1}}. Notice that ${w_{0,1}^u}(p_{{-}}) = \Gamma^1(1)$, ${w_{0,1}^s}(p_{+}) = \Gamma^2 { \setminus p_+}$, and that zeros of \eqref{eq:distance_fn} correspond to heteroclinic solutions of system \eqref{eq:ZFK_system}. 

\begin{lemma}
\label{lem:Y_u}
{For each $N \in \mathbb N$ there exists an $\eps_0 = \eps_0(N) > 0$} 
such that {$Y_u$} is $C^N$-smooth on $(c,\eps) \in (1-\beta,1+\beta) \times [0,\eps_0)$ and satisfies
\begin{equation}
\label{eq:Y_u}
{Y_u(c,\eps)} = c + \eps \int_0^\delta s^{-2} F_1(0,c,s,c) \dd s + \mathcal O(\eps^2) 
\end{equation}
as $\eps \to 0$.
\end{lemma}

\begin{proof}
Fixing $c \in (1 - \beta, 1 + \beta)$ ensures that ${w_0^u}(p_-)$, {as an orbit of \eqref{eq:ZFK_layer_problem}, intersects $\Sigma$ in a point $(1,c)$.} Regular perturbation theory {on compact subsets of $\theta<1$} implies that the fast fiber corresponding to this connection lifts to a $C^\infty$-smooth curve of the form
\begin{equation}
\label{eq:gamma_in}
\gamma^{in}(c,\eps) := \left\{ (\rho, c + \mathcal O(\eps^N), \rho^{-1} \eps ) : \eps \in [0,\eps_0] \right\} \subset \Delta_1^{in},
\end{equation}
when viewed in the $(r_1, y_1, \eps_1)$-coordinates, where $\eps_0 := \delta \rho$, see also Lemma \ref{lemma:strong}. Note that $\gamma^{in}(c,\eps) \subset \Delta_1^{in}$ as long as we choose $\delta$ and $\rho$ sufficiently small. Thus we can use {Proposition} \ref{prop:Y_smoothness} and {Lemma} \ref{lem:Y_asymptotics} in order to extend ${w_\eps^u}(p_-)$ to $\Delta_1^{out}$. Expanding in $\eps$ and using \eqref{eq:Y_asymptotics} together with the fact that $\eps = \eps_1 \rho$ for points on $\Delta_1^{in}$, we obtain
\[
{Y_u(c,\eps)} = Y \left(1 + \mathcal O(\eps^N), \rho^{-1} \eps, c \right) = 
c + \eps \int_0^\delta s^{-2} F_1(0,c,s,c) \dd s + \mathcal O(\eps^2) 
\]
as $\eps \to 0$, as required.
\end{proof}

We now derive an implicit equation for {$Y_s(c,\eps)$}.

\begin{lemma}
\label{lem:Y_s}
{For each $N \in \mathbb N$ there exists an $\eps_0 = \eps_0(N) > 0$} 
such that {$Y_s(c,\eps)$} is $C^N$-smooth on $c\in (1-\beta,1+\beta) \times [0,\eps_0)$ and satisfies
\begin{equation}
\label{eq:Y_s_eqn}
\frac{{Y_s(c,\eps)}^2 - 1}{2} - \epsilon \left[ c\delta^{-1} 	f({Y_s(c,\eps)},\delta) - \int_{\delta}^\infty \left( cs^{-2} f_1(1,s) +\frac{s^{-4}}{2} \e^{-s^{-1}}\right) \dd s\right] + \mathcal O(\epsilon^2) = 0 
\end{equation}
as $\eps \to 0$.
\end{lemma}

\begin{proof}
We begin in $K_2$ coordinates $(\eta,\theta_2,\eps)$, and look for an expression for the intersection ${w_\eps^s}(P_s) \cap \kappa_{12}(\Delta_1^{out})$, where $P_s := \{ (0,0, \eps) : \eps \in [0,\eps_0] \}$ denotes the line of equilibria emanating from $p_{+}$. We use the fact that trajectories are contained in constant level sets of the Hamiltonian function $H(\theta_2, \eta)$ defined in \eqref{eq:Hamiltonian} when $\eps = 0$, and consider now the perturbed dynamics for $\eps \in [0, \eps_0]$. We have
\[
\frac{\dd}{\dd \theta_2} H= \epsilon \left(c \eta +\frac{\theta_2}{2} \e^{\theta_2}\right),
\]
using $H=H(\theta_2, \eta(\theta_2))$.
Along the stable manifold $\eta(\theta_2;c,\epsilon) = h^s(\theta_2)+\mathcal O(\epsilon)$, where $h^s(\theta_2)$ is given by \eqref{eq:stable_manifold}, regular perturbation theory implies
\[
\frac{\dd}{\dd \theta_2} H=\epsilon \left(c h^s(\theta_2) +\frac{\theta_2}{2} \e^{\theta_2}\right)+\mathcal O(\epsilon^2) ,
\]
since we are working on compact subsets. Integrating this expression from $\theta_2=0$ (where the saddle is) to $\theta_2 = -1 / \delta$ (where $\kappa_{12}(\Delta_1^{out})$ is), we obtain
\begin{equation}\label{Hout}
H(-1/\delta, \eta^{out})-\frac12 = \epsilon \int_{0}^{-1/\delta}\left( ch^s(\theta_2) +\frac{\theta_2^2}{2} \e^{\theta_2}\right) \dd \theta_2+\mathcal O(\epsilon^2) ,
\end{equation}
where $\eta^{out}$ denotes the $\eta$-coordinate of the intersection ${w_\eps^s}(P_s) \cap \kappa_{12}(\Delta_1^{out})$ to be solved for. 
Using \eqref{y1def}, we have 
\begin{align*}
H(-1/\delta, \eta^{out})&=\frac{(y_1^{out})^2}{2}-\frac{\partial}{\partial \eta} H (-1/\delta, f_1(y_1^{out}))cr_1^{out}+\mathcal O((r_1^{out})^2)\\
&=\frac{(y_1^{out})^2}{2}-cr_1^{out}f_1(y_1^{out},\delta)+\mathcal O((r_1^{out})^2)
\end{align*}
where $y_1^{out} = {Y_s}(c,\eps)$ is given in terms of $\eta^{out}$ by the right-hand side of \eqref{this} evaluated at $(r_1^{out}, \eta^{out}, \delta)$, where $r_1^{out} = \delta^{-1} \epsilon$. Using \eqref{f1prop} and \eqref{Hout}, we obtain the equation
\[
\frac{(y_1^{out})^2}{2} -\frac12 - \epsilon \left[ c\delta^{-1} f(y_1^{out},\delta) - \int_{\delta}^\infty \left( cs^{-2} f_1(1,s) +\frac{s^{-4}}{2} \e^{-s^{-1}}\right) \dd s+\mathcal O(\epsilon)\right] = 0 ,
\]
as required.
\end{proof}

In order to identify the zeros of $\mathcal D(c,\eps)$, it suffices to consider solutions to the bifurcation equation
\begin{equation}
\label{eq:B_tilde}
\widetilde B(c, \eps) := B({Y_u}(c,\eps), c, \eps) = 0 ,
\end{equation}
where the function $B(X,c,\eps)$ is defined using the implicit equation \eqref{eq:Y_s_eqn}, i.e.~
\begin{equation}
\label{eq:B}
B(X, c, \eps) := \frac{X^2 - 1}{2} - \epsilon \left[c\delta^{-1} f(X,\delta) - \int_{\delta}^\infty \left( cs^{-2} f_1(1,s) +\frac{s^{-4}}{2} \e^{-s^{-1}}\right) \dd s\right] + \mathcal O(\epsilon^2) .
\end{equation}
The following result establishes a solution to \eqref{eq:B_tilde}.

\begin{lemma}
\label{lem:heteroclinic_minimal}
{For each $N \in \mathbb N$ there exists an $\eps_0 = \eps_0(N) > 0$ and} a $C^N$-smooth function {$\overline c : [0,\eps_0) \to [1, \infty)$} such that $\widetilde B(\overline c(\eps), \eps) = 0$ for all $\eps \in [0,\eps_0)$. In particular,
\[
{\overline c(0) = {1} ,\quad \overline c'(0) =  \int_{-\infty}^0 \left(1-h^s(x)\right) \dd x-1 \approx 0.34405>0 . }
\]
\end{lemma}

\begin{proof}
Using \eqref{eq:B_tilde} and \eqref{eq:B}, we obtain $\widetilde B(1,0) = 0$ and $\widetilde B_c'(1,0) = 1 \neq 0$. Thus, the existence of the $C^N$-smooth function $\eps \mapsto \overline c(\eps)$ satisfying $\overline c(0) = 1$ follows from the $C^N$-smoothness of $\widetilde B$ and the implicit function theorem.

It remains to calculate $\overline c'(0)$. Since $\overline c'(0) = - \widetilde B'_\eps(1,0) / \widetilde B'_c(1,0) = -\widetilde B'_\eps(1,0)$ by implicit differentiation, it suffices to determine $\widetilde B'_\eps(1,0)$. 
From \eqref{eq:B_tilde}, using the expression for {$Y_u(c,\eps)$} in \eqref{eq:Y_u}, we obtain
\[
\widetilde B_\eps'(1, 0) = 
\int_0^\delta s^{-2} F_1(0,1,s,1) \dd s - \delta^{-1} f(1,\delta) + \int_{\delta}^\infty \left( s^{-2} f_1(1,s) +\frac{s^{-4}}{2} \e^{-s^{-1}}\right) \dd s ,
\]
where we also used the fact that $\overline c(0) = 1$. In order to simplify the right-hand side, we use the definitions of $f_1$ and $F_1$ given in \eqref{this} and \eqref{eq:F1}, together with $\int_{0}^\infty \frac{s^{-4}}{2} \e^{-s^{-1}} \dd s = 1$. We obtain
\begin{equation}
\label{eq:tilde_B_eps}
\widetilde B_\eps'(1, 0) = 
1 + \left[ - \delta^{-1} f_1(1,\delta) + \frac{1}{2} \int_0^\delta \frac{s^{-4} \e^{-s^{-1}}}{f_1(1,s)} \dd s +	 \int_\delta^\infty s^{-2} f_1(1,s) \dd s \right] .
\end{equation}
Differentiating with respect to $\delta$ shows that \eqref{eq:tilde_B_eps} is independent of $\delta$. Therefore,
\begin{equation*}
\begin{split}
\widetilde B_\eps'(1, 0) &= 
\lim_{\delta \to 0} \widetilde B_\eps'(1, 0) \\
&= 1 + \lim_{\delta\rightarrow 0 } \left(\int_{\delta}^\infty s^{-2} f_1(1,s) \dd s - \delta^{-1} f_1(1,\delta)\right) \\
&= 1 + \int_0^{\infty} s^{-2} \left(f_1(1,s)-f_1(1,0)\right) \dd s \\
&= 1 + \int_{-\infty}^0  \left(h^s(\theta_2)-1\right) \dd \theta_2,
\end{split}
\end{equation*}
as required. 
\end{proof}

\subsection{Completing the proof of Theorem \ref{thm:main}}

Lemma \ref{lem:heteroclinic_minimal} shows the existence of a $C^N$-smooth minimal wave speed {function $\overline c(\eps)$} satisfying 
\eqref{eq:wave_speed} and {Assertion (ii)} of Theorem \ref{thm:main}. {To see that $\overline c$ is $C^\infty$, we first notice that $\overline c$ is defined on a domain $[0,\epsilon_0)$ with $\epsilon_0>0$ small enough (we can apply Lemma \ref{lem:heteroclinic_minimal} with $N=1$). Clearly $\overline c$ is $C^\infty$ for any $\epsilon\in (0,\epsilon_0)$, since the {original ZFK system \eqref{eq:ZFK_system}} is {$C^\infty$}. We therefore only have to show that $\overline c$ is $C^\infty$ at $\epsilon=0$. For this purpose, we use the uniqueness of $\overline c(\epsilon)$ for $\epsilon\in (0,\epsilon_0)$ and Lemma \ref{lem:heteroclinic_minimal} to conclude that $\overline c$ is $C^N$ on $[0,\epsilon_0(N))$ for some $0<\epsilon_0(N)\ll 1$ for any $N\in \mathbb N$.} The claim now follows.

{Next, we notice that {all heteroclinic} connections with $c\ne\overline c(\epsilon)$ are weak{, due to the uniqueness of the strong heteroclinic connection}. In particular, for $c<\overline c(\epsilon)$ the heteroclinic connections enter $\theta<0$. This shows {Assertion (i)} of Theorem \ref{thm:main}. To prove {Assertion (iii)}, we just have to show that the connections limit to $\Gamma(c)$ in Hausdorff distance. For this purpose, we first notice that for $c>\overline c(\epsilon)$ the stable manifold $w_\epsilon^s(p_{+})$ lies below the strong unstable manifold $w_\epsilon^u(p_{-})$. Moreover, $w_0^s{(p_+)}$ and {$\Gamma^1(c)$} (as a trajectory of \eqref{eq:ZFK_layer_problem}) have ${q} : (r_1,\eta,\epsilon_1)=(0,1,0)$ as their $\alpha$-limit resp. $\omega$-limit sets in the $K_1$-chart{, as shown in Figure \ref{fig:blow-up}}. It then follows from {Proposition} \ref{prop:Y_smoothness} that the intersection of $w_\epsilon^s(p_{+})$ with $\Delta^{in}=\{(\theta,\eta):\theta=1-\rho\}$ is $\mathcal O(\epsilon)$-close to $\Gamma^2(c)\cap \Delta^{in}$. Here we abuse notation slightly by using $\Delta^{in}$ for the blown-down version of \eqref{eq:sections}. The result then follows from Fenichel theory and Proposition \ref{prop:slow_manifold}.}

\section{Summary and outlook}
\label{sec:summary_and_outlook}

A primary aim of this work was to demonstrate the suitability of GSPT and geometric blow-up for the study of dynamical systems with singular exponential nonlinearities. We would like to emphasise and reiterate the following methodological point, which is supported by our analysis of the ZFK equation \eqref{eq:ZFK_eqn} as well as the analysis of electrical oscillator models in \cite{Jelbart2019c}:

\smallskip

\textit{Smooth dynamical systems with {singular} exponential nonlinearities can often be formulated as singular perturbation problems with a piecewise-smooth singular limit, after which the geometry and dynamics can be analysed using GSPT and geometric blow-up.}

\smallskip

This means that different limiting problems are to be expected on different regions of the phase space which are separated by 
switching manifolds, along which the system loses smoothness. Because the system is smooth for $\eps > 0$, one expects to be able to find a smooth connection between these regions by inserting an inner `rescaling regime', where another limiting problem arises, after geometric blow-up; we refer again to  \cite{Kristiansen2019c}. 
In the case of system \eqref{eq:ZFK_system}, 
the {singular exponential nonlinearity in the reaction term $\omega(\theta, \eps^{-1})$} leads to a non-smooth singular limit {in a `normalised system' which is obtained from \eqref{eq:ZFK_system} by dividing the the right-hand side by the positive expression in \eqref{eq:quantity}, since
	\[
	\frac{\omega(\theta, \eps^{-1})}{\eps^{-2} \left( 1 + \frac{1}{2} \e^{\eps^{-1} (1 - \theta)} \right)} \ \SJ{\to} \ 
	\begin{cases}
		0, & \theta < 1 , \\
		\theta (1 - \theta) , & \theta > 1 ,
	\end{cases}
	\]	
	as $\eps \to 0$. This leads to different limiting problems on three distinct regimes: $\theta < 1$, $\theta = 1 + \mathcal O(\eps)$ and $\theta > 1$. In order to identify and characterise travelling wave solutions to the ZFK equation \eqref{eq:ZFK_eqn}, it sufficed to consider the dynamics in two regimes only, namely, in the} 
{convective}-diffusive zone with $\theta < 1$ bounded away from $\Sigma = \{ \theta = 1 \}$, and a reactive-diffusive zone with $\theta = 1 + \mathcal O(\eps)$. After extending the phase space and applying the blow-up transformation \eqref{eq:blow-up}, we were able to {smoothly} connect the analyses in each of these zones, which we performed separately in Section \ref{sec:geometric_singular_perturbation_analysis}. The connection problem in chart $K_1$ within the blow-up is smooth, even though the singular limit in the original system \eqref{eq:ZFK_system} is not, showing that the blow-up has `resolved' the loss of smoothness.

Our main result is Theorem \ref{thm:main}, which describes a $(c,\eps)$-family of heteroclinic orbits in system \eqref{eq:ZFK_system} which correspond to travelling waves the original ZFK equation \eqref{eq:ZFK_eqn}. In essence, Theorem \ref{thm:main} is the rigorous and geometrically informative counterpart of existing results on travelling waves in (a number of closely related variants of) the ZFK equation, which have been obtained formally using HAEA 
{in e.g.~\cite{Buckmaster1982,Bush1979,Clavin2016,Williams1985,Zeldovich1938}}. 
We also moved beyond the existing formal results by proving {that} the minimal wave speed function $\overline c(\eps)$ {is $C^\infty$ for all $\eps \in [0,\eps_0)$} ({Proposition} \ref{prop:Y_smoothness} and the normal form in \cite{DeMaesschalck2016} was crucial for this), and providing an asymptotic expansion for the slow manifold $S_\eps$ (see again Proposition \ref{prop:slow_manifold}) which plays an important role in the construction of the non-minimal waves for $c > \overline c(\eps)$. The latter feature is not likely to be of significance for the ZFK problem in particular, given that one expects the system to select the minimal wave speed $\overline c(\eps)$ (recall Remark \ref{rem:c_selection}). Nevertheless, the proof of Proposition \ref{prop:slow_manifold} {is} of independent interest because it provides a rigorous and dynamical systems based approach to the determination of asymptotic expansions for flat slow manifolds.

This work paves the way for more complicated analyses \PS{of systems with singular exponential nonlinearities. 
	With regard to problems in combustion theory, a natural continuation of this work would be to consider the case of planar flames with non-unity Lewis number. In this case, one has to carry out a similar analysis for a system of two reaction-diffusion equations with different diffusivity constants, see e.g. 
	\cite[Eq.~(2.1.4)]{Clavin2016}.}
We also left questions of stability and selection completely open, see e.g.~Remark \ref{rem:c_selection}.  These and other related problems remain for future work.

\bibliographystyle{siam}
\bibliography{ZFK_bib}

\end{document}